\documentclass[a4paper,reqno,11pt]{amsart}

\usepackage{amssymb}
\usepackage{amsthm}
\usepackage[style=alphabetic, url=false, isbn = false, backend=biber]{biblatex}
\usepackage{tikz-cd}
\usepackage{mathtools}
\usepackage{framed}
\usepackage[mathscr]{eucal}
\usepackage{fullpage}
\usepackage{hyperref}
\usepackage{enumerate}
\usepackage{wasysym}
\usepackage{xcolor}
\usepackage{scalerel}
\usepackage{tikz}
\usetikzlibrary{arrows.meta}

\hypersetup{colorlinks=true,citecolor=blue,urlcolor =black,linkbordercolor={1 0 0}}
\mathtoolsset{centercolon}

\newtheorem{theorem}{Theorem}

\newtheorem{proposition}[theorem]{Proposition}
\newtheorem{lemma}[theorem]{Lemma}
\newtheorem{corollary}[theorem]{Corollary}
\newtheorem{conjecture}[theorem]{Conjecture}

\theoremstyle{definition}
\newtheorem{defn}[theorem]{Definition}

\newtheorem{situation}[theorem]{Situation}
\newtheorem*{warning}{Warning}

\theoremstyle{remark}
\newtheorem*{remark}{Remark}

\newcommand{\on}{\operatorname}
\newcommand{\abs}[1]{\left|#1\right|}

\renewcommand{\tilde}[1]{\widetilde{#1}}

\newcommand{\Db}{\mathrm{D}^{\mathrm{b}}}

\newcommand{\im}{\on{im}}

\newcommand{\coker}{\on{coker}}
\newcommand{\Hom}{\on{Hom}}

\newcommand{\Sym}{\on{Sym}}
\renewcommand{\bar}[1]{\overline{#1}}

\newcommand{\cO}{\mathcal{O}}

\newcommand{\cE}{\mathcal{E}}

\newcommand{\cK}{\mathcal{K}}

\newcommand{\cN}{\mathcal{N}}

\newcommand{\cQ}{\mathcal{Q}}
\newcommand{\cR}{\mathcal{R}}

\newcommand{\cU}{\mathcal{U}}

\newcommand{\cW}{\mathcal{W}}
\newcommand{\cX}{\mathcal{X}}
\newcommand{\cZ}{\mathcal{Z}}

\newcommand{\bZ}{\mathbf{Z}}

\newcommand{\bbC}{\mathbb{C}}

\newcommand{\bbL}{\mathbb{L}}
\newcommand{\bbP}{\mathbb{P}}

\newcommand{\bbZ}{\mathbb{Z}}

\newcommand{\fQ}{\mathfrak{Q}}

\newcommand{\Bl}{\on{Bl}}
\newcommand{\Gr}{\on{Gr}}
\newcommand{\Fl}{\on{Fl}}
\newcommand{\OGr}{\on{OGr}}

\newcommand{\Span}[1]{\langle#1\rangle}
\newcommand{\Pic}{\on{Pic}}

\newcommand{\cl}{\mathcal{C}l}

\newcommand{\longdashleftrightarrow}[1][1.75em]{\mathrel{
		\tikz[baseline]{\draw[dash pattern=on .25em off .1 em,<->](0,.58ex)--(#1,.58ex)}}}

\title{Flips for spaces of quadrics on del Pezzo varieties} 
\author{Saket Shah}

\begin{document}
	\maketitle
	\begin{abstract}
		For a cubic hypersurface $X$, work of Galkin--Shinder and Voisin shows the existence of a birational map relating the Hilbert scheme of two points $X^{[2]}$ with a certain projective bundle over $X$. Belmans--Fu--Raedschelders show that this is a standard flip, a particularly nice type of birational map inducing decompositions of derived categories. \par 
		We show that this geometric construction extends to produce standard flips for Hilbert schemes of quadrics on various higher-dimensional del Pezzo varieties of degree at least 3, including cubics, intersections of two quadrics, and linear sections of $\Gr(2,5)$. The resulting construction also generalizes results of Chung--Hong--Lee for quintic del Pezzo varieties. \par 
		As an application, we produce a conjectural semiorthogonal decompositions for orthogonal Grassmannians of lines. 
	\end{abstract}
	\tableofcontents
	\section{Introduction}
	\subsection*{Conventions}
	We work throughout over the complex field $\bbC$. 
	\subsection{The flips}
	Fix for now an $(n+1)$-dimensional vector space $V$. For a cubic hypersurface $X \subset \bbP(V)$, there is a beautiful construction of a birational map that can be associated to the Hilbert square $X^{[2]}$ of $X$ which has been called the \textit{Galkin--Shinder--Voisin flip} \cite{gsflip, voisin}. \par 
	To see how such a map arises, one observes that the rational map $X^{[2]} \dashrightarrow X$ realizing the classical third-point construction for cubic hypersurfaces lifts to a birational map between $X^{[2]}$ and the projectivization of a certain vector bundle on $X$. This gives rise to the following diagram: 
	\begin{center}
		\begin{tikzcd}
			\bbP(\Sym^2 \cU_2^\vee) \arrow[r, hook, "\iota"] \arrow[d, "\pi"] & X^{\left[2\right]} \arrow[leftrightarrow, r, dashed, "\text{flip}"] & \bbP(\cQ_1) \arrow[d] & \bbP(\cU_2^\vee) \arrow[d] \arrow[l,hook']\\
			F_1(X) & & X & F_1(X)
		\end{tikzcd}
	\end{center}
	In the above diagram:
	\begin{itemize}
		\item $F_1(X)$ denotes the Fano variety of lines on the cubic $X$,
		\item $\cU_2$ on $F_1(X)$ denotes the tautological rank 2 subbundle of $V \otimes \cO$ pulled back from the Grassmannian, 
		\item the embedding $\iota : \bbP(\Sym^2 \cU_2^\vee) \hookrightarrow X^{[2]}$ is simply the inclusion of those length 2 subschemes on $X$ which live on a line in $X$, and 
		\item $\cQ_1$ on $X$ denotes the tautological rank $n$ quotient bundle of $V \otimes \cO$ pulled back from $\bbP(V) = \Gr(1,V)$.
	\end{itemize}
	This birational map turns out to be a \textit{standard flip}. Standard flips are particularly nice instances of flips arising in the MMP, in particular because they are known to imply a motivic relationship between the two birational models. For example, the standard flip above induces a relationship on the level of the Grothendieck ring of varieties which was studied in detail by Galkin and Shinder in \cite{gsflip}. On the other hand, one can show that any standard flip induces a semiorthogonal decomposition of derived categories, and hence one can use the above picture to produce an interesting semiorthogonal decomposition of $\Db(X^{[2]})$ in terms of $\Db(F_1(X))$ and several copies of $\Db(X)$, as in \cite{cubicflip}. \par 
	The first goal of this paper is to produce a twofold generalization of the flips above. The directions along which we will generalize are as follows: 
	\begin{enumerate}
		\item One can enlarge the class of varieties $X$ for which these flips exist to other varieties besides cubic hypersurfaces. The class we will consider will be the class of \textit{del Pezzo varieties}, which have been explored at length in \cite{fujitaI,fujitaII,fujitaIII,kuzdelpezzo}. 
		\item One can interpret $X^{[2]}$ as the space of 0-dimensional quadrics on $X$, and can then consider the analogous Hilbert scheme $G_k(X)$ of $k$-dimensional quadrics on $X$ for any $k \geq 0$. This approach is similar in spirit to that of \cite{GMquadrics}, where similar birational maps were constructed for moduli spaces of quadrics on Gushel--Mukai varieties (which are in turn quadric hypersurface sections of del Pezzo varieties of degree 5). 
	\end{enumerate}
	\par 
	Our first result, which we prove in Section \ref{contractionflipssection}, arises from the first generalization given a restriction on the allowed degrees of the del Pezzo varieties. 
	\begin{theorem}
		\label{delpezzoflip}
		Suppose that $(X,H)$ is an $n$-dimensional del Pezzo variety with $3 \leq d(X) \leq 9$. Then the polarization $H$ is very ample, and there exists a standard flip diagram
		\begin{equation}
			\label{flipdiag}
			\begin{tikzcd}
				\bbP(\Sym^2 \cU_2^\vee) \arrow[r, hook, "\iota"] \arrow[d, "\pi"] & X^{\left[2\right]} \arrow[leftrightarrow, r, dashed, "\text{flip}"] & Y \\
				F_1(X)
			\end{tikzcd}
		\end{equation}
		where $Y$ is a smooth projective variety. In particular, there is always an embedding $\Db(F_1(X)) \hookrightarrow \Db(X^{[2]})$ as a semiorthogonal component. 
	\end{theorem}
	Then our second theorem studies the spaces of higher-dimensional quadrics $G_k(X)$ where $X$ is a del Pezzo variety of degree 3, 4, or 5, which are the prime cases of interest. One should interpret this as giving a geometric construction for the standard flip above. We state the theorem in broad terms first, deferring more precise statements and their proofs to Section \ref{geomflipsection} in the paper. 
	
	\begin{theorem}[Theorems \ref{geomflipcubics}, \ref{geomflipquartic} and \ref{geomflipquintic}]\label{geomflipthm}
		Let $X$ be a smooth del Pezzo variety of dimension $n$, and consider the Hilbert scheme $G_k(X)$ of $k$-dimensional quadrics on $X$. Then under the appropriate numerical conditions on $n$ and $k$, we have the following flips: 
		\begin{enumerate}[(a)]
			\item Suppose $X \subset \bbP^{n+1}$ is generic smooth cubic hypersurface, and let $\cQ_{k+1}$ denote the tautological quotient bundle of rank $(n+2) - (k+1)$ on $F_k(X)$. Then we have a standard flip: 
			\begin{equation}
				\begin{tikzcd}
					G_k(X) \arrow[leftrightarrow, r, dashed, "\text{flip}"] & \bbP(\cQ_{k+1}) \arrow[d] \\
					& F_k(X) 
				\end{tikzcd}
			\end{equation}
			\item Suppose $X = Q_1 \cap Q_2 \subset \bbP^{n+2}$ is any smooth complete intersection of two quadric hypersurfaces. Let $\fQ$ denote the total space of the quadric fibration for the pencil of quadrics $\langle Q_1, Q_2 \rangle$, and write $\OGr(k+2,\fQ)$ for the relative orthogonal Grassmannian of isotropic $(k+2)$-dimensional subspaces of $V$ for the quadric fibration. Then we have a standard flip 
			\begin{equation}
				\begin{tikzcd}
					G_k(X) \arrow[leftrightarrow, r, dashed, "\text{flip}"] & \OGr(k+2,\fQ) \arrow[d] \\
					&\langle Q_1, Q_2 \rangle = \bbP^1 
				\end{tikzcd}
			\end{equation}
			\item Let $V_5$ be a 5-dimensional vector space, and suppose $X = \Gr(2,V_5) \cap \bbP(W) \subset \bbP(\bigwedge^2 V_5)$ is a smooth transverse linear section for some $\bbP(W) \subset \bbP(\bigwedge^2 V_5)$. Suppose $\cU_{k+2}$ is the tautological subbundle for $\Gr(k+2,\bigwedge^2 V_5)$ and $\cR_4$ is the tautological subbundle for $\Gr(4,V_5)$. Then for $k = 0, 1$ we have a standard flip: 
			\begin{equation}
				\begin{tikzcd}
					G_k(X) \arrow[leftrightarrow, r, dashed, "\text{flip}"] & \Gr_{\Gr(4,V_5)}(k+2,\bigwedge^2 \cR_4 \cap W) \arrow[d] \\
					&\Gr(4,V_5) 
				\end{tikzcd}
			\end{equation}
			When $k = 2$, we instead have \begin{equation}
				G_2(X) = \Gr_{\Gr(4,V_5)}(4, \bigwedge^2 \cR_4 \cap W) \sqcup \bbP_{F_3(X)}(\Sym^4 \cU_{4}^\vee)
			\end{equation}
			is a disjoint union of two smooth irreducible subvarieties. When $k \geq 2$, $G_k(X)$ is isomorphic to the Grassmann bundle $\Gr_{\Gr(4,V_5)}(k+2, \bigwedge^2 \cR_4 \cap W)$.
		\end{enumerate}
	\end{theorem}
	\begin{remark}
		\begin{itemize}
			\item Taking $k = 0$ in part (a) recovers the flip of Galkin--Shinder and Voisin \cite{gsflip, voisin}. 
			\item Part (b) of this theorem was suggested in Remark A.8 of the appendix by Kuznetsov to \cite{colkuz}, and there is some overlap in the intermediate results. For example, Kuznetsnov shows that $G_k(X)$ and $\OGr(k+2,\fQ)$ are birational in Proposition A.7. In \cite{nonclosedfields} this description was given for the complete intersection of two quadrics in $\bbP^5$. 
			\item Part (c) generalizes results of \cite{chunghonglee, conicsquinticdelpezzo} for conics in linear sections of the Pl\"ucker embedding of $\Gr(2,5)$ in $\bbP^9$, and our methods are similar (though we avoid discussing clean intersections for the blowup center, and instead show the center is smooth by explicitly identifying it as the projectivization of a vector bundle). Moreover, the results of part (c) are closely related to similar results for Gushel--Mukai varieties in \cite{GMquadrics}.
			\item Beyond these cases, similar methods prove that for a smooth quadric $Q \subset \bbP(V)$, 
			\[G_k(Q) = \Bl_{\OGr(k+2,Q)}\Gr(k+2,V).\]
			More generally, for a variety $X \subset \bbP(V)$ cut out by a space of quadrics $W \subset \Sym^2 V^\vee$, one should expect an isomorphism between an open subset of $G_k(X)$ and an open subset of the relative Fano scheme $F_{k+1}(\cX/\bbP(W^\vee))$, where the fiber of $\cX \to \bbP(W^\vee)$ over $\phi \in W^\vee$ is cut out by the quadrics in $\ker \phi$. 
		\end{itemize}
	\end{remark}
	While Theorem \ref{geomflipthm} is a strictly stronger result than Theorem \ref{delpezzoflip} for del Pezzo varieties of degree 3, 4 or 5, the proof of Theorem \ref{delpezzoflip} is completely uniform in the degree. On the other hand, while each part of Theorem \ref{geomflipthm} falls within the same framework, the details of the proofs differ. \par 
	The methods in this paper are inspired by those of \cite{GMquadrics}. For a del Pezzo variety $X \subset \bbP(V)$ in the cases of Theorem \ref{geomflipthm} above, the general strategy we use is to exploit the natural morphism from $G_k(X) \to \Gr(k+2,V)$ taking a quadric $\Sigma \subset X$ to its linear span $\langle \Sigma \rangle$. We then find a smooth projective variety $Y'$ which also admits a natural map to $\Gr(k+2,V)$ and show moduli-theoretically that there is an isomorphism on open sets of $G_k(X)$ and $Y'$ over $\Gr(k+2,V)$, which gives us a large smooth open subset of $G_k(X)$. Using deformation theory, we show that $G_k(X)$ is smooth on the complement of this open set, and then complete the argument by showing that the blowup of $Y'$ on the complement of this open set naturally admits a blowdown to $G_k(X)$ --- except in one case, where $G_k(X)$ is disconnected. \par 
	\subsection{Derived categories}
	As an immediate corollary of Theorem \ref{geomflipthm}, we deduce: 
	\begin{corollary} \label{sodcorollary}
		Under the hypotheses of the theorem above, we have semiorthogonal decompositions (where we have suppressed the embedding functors):
		\begin{enumerate}
			\item When $X \subset \bbP^{n+1}$ is a cubic hypersurface, 
			\begin{equation}
				\begin{aligned}
					\Db(G_k(X)) &= \langle\Db(\bbP_X(\cQ_{k+1})), \underbrace{\Db(F_{k+1}(X)),\dots,\Db(F_{k+1}(X))}_{{k+3 \choose 2} - (k+2)\text{ times}}\rangle \\
					&= \langle\underbrace{\Db(F_k(X)),\dots,\Db(F_k(X))}_{n-k+2\text{ times}}, \underbrace{\Db(F_{k+1}(X)),\dots,\Db(F_{k+1}(X))}_{{k+3 \choose 2} - (k+2)\text{ times}}\rangle.
				\end{aligned}
			\end{equation}
			\item When $X \subset \bbP^{n+2}$ is a complete intersection of two quadrics, 
			\begin{equation}
				\begin{aligned}
					\Db(G_k(X)) &= \langle\Db(\OGr(k+2,\fQ)), \underbrace{\Db(F_{k+1}(X)),\dots,\Db(F_{k+1}(X))}_{{k+3 \choose 2} - 2\text{ times}}\rangle. \\
				\end{aligned}
			\end{equation}
			\item When $X$ is a linear section of $\Gr(2,V_5)$, $\Db(G_k(X))$ admits a full exceptional collection. 
		\end{enumerate}
	\end{corollary}
	\begin{proof}
		For cubics and complete intersections of quadrics, the claims follow immediately from \cite[Theorem A]{cubicflip} and Theorem \ref{geomflipthm}. \par 
		For linear sections of Grassmannians, \cite[Theorem A]{cubicflip} combined with Theorem \ref{geomflipthm} implies that $\Db(G_k(X))$ admits a decomposition whose components are the derived category of $\Gr_{\Gr(4,V_5)}(k+2,\cR)$ and the derived category of (connected components of) $F_{k+1}(X)$. The first is a Grassmann bundle over a projective space and therefore has a full exceptional collection. By the classification in \cite[Lemma 3.2]{GMquadrics}, every component of $F_{k+1}(X)$ is either a type A partial flag variety, a smooth quadric, or a blowup of a type A partial flag variety along either a linear subspace or a smooth quadric; by combining Orlov's blowup formula \cite{orlovblowupformula} with Kapranov's exceptional collections on type A partial flag varieties and on smooth quadrics the existence of the full exceptional collection follows. 
	\end{proof}
	In Section \ref{application}, we present an application of these geometric constructions to finding new semiorthogonal decompositions. Using the case of Theorem \ref{geomflipthm} for complete intersections of two quadrics along with a conjectural decomposition of $\Db(F_1(X))$ from \cite{fanoschemequaddecomp}, we use categorical heuristics to propose and give evidence for a conjectural semiorthogonal decomposition for derived categories of some fibrations in orthogonal Grassmannians, allowing for psosibly degenerate quadratic forms. \par 
	Finally, in Section \ref{counterexample} we show that the Theorems \ref{delpezzoflip} and \ref{geomflipthm} are sharp in the degree. To be more precise, we show that for a smooth del Pezzo threefold of degree 2 --- a quartic double solid --- there is an obstruction to the existence of any flip that would induce a semiorthogonal decomposition of $\Db(X^{[2]})$ including $\Db(F_1(X))$ as a semiorthogonal component. 
	\subsection{Acknowledgments}
	The author would like to thank David Stapleton and Riku Kurama for interesting and useful conversations, as well as Pieter Belmans and Alexander Kuznetsov who read a preliminary version of this paper. Special thanks go to Alexander Perry for his continued support. This material is based upon work supported by the National Science Foundation under NSF RTG grant DMS-1840234 and NSF FRG grant DMS-2052750.
	\section{Background}
	\subsection{Del Pezzo varieties}
	Recall the following generalization of del Pezzo surfaces to varieties of dimension $\geq 2$. 
	\begin{defn}
		A \emph{del Pezzo variety} is a polarized smooth variety $(X,H)$ of dimension $n$ such that $-K_X = (n-1)H$.
	\end{defn} 
	In fact for such a given variety $X$, there is a unique choice of such a polarization $H$. The following proposition can be checked by induction on the dimension of $X$. 
	\begin{proposition}[Corollary 3.4 of \cite{kuzdelpezzo}]
		Suppose that $(X,H)$ is a del Pezzo variety as defined above. Then $\Pic(X)$ is a free abelian group. In particular, if there exists another polarization $H'$ such that $-K_X = (n-1)H'$, it follows that $H = H'$. 
	\end{proposition}
	As such, it is reasonable to omit the polarization $H$ when discussing a given del Pezzo variety $X$, since all of the geometry of the pair $(X,H)$ is in fact intrinsic to $X$. \par 
	The following proposition is a classical result, and gives a numerical criterion for the global generation of the linear system $\abs{H}$ based on the degree $d(X) \coloneq (H)^n$. 
	\begin{proposition}[Proposition 2.2 of \cite{kuzdelpezzo}] 
		Suppose $X$ is a del Pezzo variety. Then if $d(X) \geq 2$, the linear system $\abs{H}$ is base-point-free and the morphism $X \to \abs{H}$ is finite onto its image. If $d(X) \geq 3$, the linear system $\abs{H}$ is very ample. 
	\end{proposition}
	A more precise classification result by the degree $d(X)$ is also possible, by the work of Fujita \cite{fujitaI,fujitaII,fujitaIII}, though we will use it as stated in \cite{kuzdelpezzo}. 
	\begin{theorem}[Theorem 1.2 of \cite{kuzdelpezzo}]
		Suppose that $X$ is a del Pezzo variety of dimension $n$. Then: 
		\begin{enumerate}[(i)]
			\item If $d(X) = 1$, then $X$ is a sextic hypersurface avoiding singular points in $\bbP(1^n, 2,3)$. 
			\item If $d(X) = 2$, then $X$ is a double covering of $\bbP^n$ branched over a quartic hypersurface via the linear system $\abs{H}$. 
			\item If $d(X) = 3$, then $X$ is a cubic hypersurface in $\bbP^{n+1}$.
			\item If $d(X) = 4$, then $X$ is a complete intersection of 2 quadric hypersurfaces in $\bbP^{n+2}$. 
			\item If $d(X) = 5$, then $X$ is a linear section of $\Gr(2,5) \subset \bbP^9$ embedded via the Pl\"ucker embedding, with $2 \leq \dim X \leq 6$. 
			\item If $d(X) = 6$, then either $X$ is a linear section of $(\bbP^1)^3 \subset \bbP^7$ or $(\bbP^2)^2\subset \bbP^8$ embedded via the Segre embeddings. 
			\item If $d(X) = 7$, then either $X$ is a del Pezzo surface of degree 7 or $X$ is a blowup of $\bbP^3$ at a single point. 
			\item If $d(X) = 8$, then either $X \simeq \bbP^3$ or $X$ is a del Pezzo surface of degree 8. 
			\item If $d(X) = 9$, then $X$ is $\bbP^2$. 
		\end{enumerate}
	\end{theorem}
	For the majority of this work, we focus on the case where $d(X) \geq 3$, or equivalently where the natural polarization is very ample. 
	\subsection{Hilbert schemes}
	We use the conventions and notation for Hilbert schemes presented in \cite[Section 3]{GMquadrics}. In particular for a polarized variety $(X,H)$ we let $F_k(X)$ denote the Hilbert scheme associated to the Hilbert polynomial 
	\[\frac{(t+1)\cdots(t+k)}{k!}.\]
	It is often called the \textit{Fano variety} or \textit{Fano scheme} of $k$-planes on $X$. When $X =  \bbP(V)$, this is identified with the Grassmannian $F_k(X) \simeq \Gr(k+1,V)$. We also let $G_k(X)$ denote the Hilbert scheme associated to the Hilbert polynomial 
	\[\frac{(t+1)\cdots(t+k-1)}{k!}(2t-k)\]
	and call it the Hilbert scheme of $k$-dimensional quadrics on $X$. The key observation used in \cite{GMquadrics} as well as in this paper is that any $k$-dimensional quadric $\Sigma$ is a quadric hypersurface in a linear space $\langle \Sigma \rangle$ of dimension $k + 1$, which we call its linear span. In particular there is a natural projective bundle morphism 
	\[G_k(\bbP(V)) \simeq \bbP(\Sym^2 \cU_{k+2}^\vee) \to \Gr(k+2,V).\]
	Moreover when the map induced by the linear system $H$ is an embedding $X \hookrightarrow \bbP(V)$ we have natural embeddings $F_k(X) \hookrightarrow F_k(\bbP(V)) \simeq \Gr(k+1,V)$ and $G_k(X)  \hookrightarrow G_k(\bbP(V))$, so in particular one has a natural map 
	\[G_k(X) \to \Gr(k+2,V)\]
	taking $\Sigma$ to $\langle \Sigma \rangle$. 
	\subsection{Vector bundles on Hilbert squares}
	For this section, let $X$ be any smooth variety, and take a positive integer $m$. Then to any vector bundle $\cE$ on $X$, it is possible to associate a tautological bundle $\cE^{[m]}$ on the Hilbert scheme of $m$ points $X^{[m]}$. Indeed, the Hilbert scheme $X^{[m]}$ comes with a universal closed subscheme $\cZ_m$, amounting to the following diagram: 
	\begin{center}
		\begin{tikzcd}
			& \cZ_m \arrow[rd, "q"] \arrow[ld, "p"']& \\
			X & & X^{[m]},
		\end{tikzcd}
	\end{center}
	We then define $\cE^{[m]} := q_*p^*\cE$. \par
	When $m = 2$, the Hilbert scheme and its tautological bundles are particularly easy to study. In this case it is well known that both $X^{[2]}$ and $\cZ_2$ are smooth; in fact, we can take $\cZ_2$ to be the blowup $\Bl_{\Delta} (X \times X)$ along the diagonal $\Delta\subset X\times X$, $p$ to be the composition of the blowup map with the first (or second) projection, and $q$ to be the $\bZ/2$-quotient by the natural involution. In particular, the functor $(-)^{[2]}$ is exact since $p$ is flat by miracle flatness and $q$ is finite flat. \par 
	There is a readily available description of these tautological vector bundles in terms of the $\bZ/2$-equivariant geometry of this double covering. The following canonical exact sequence appears in the proof of \cite[Proposition 2.3]{danila}, though the proof we present here differs slightly. 
	\begin{proposition}
		\label{sequenceprop}
		Suppose that $\cE$ is a vector bundle on $X$. Then there is a $\bZ/2$-equivariant sequence of vector bundles on $\cZ_2 = \Bl_{\Delta} (X \times X)$ given by 
		\begin{equation}
			0 \to q^*\cE^{[2]} \to p_1^*\cE \oplus p_2^*\cE \to \iota_*\pi^*\cE \otimes \chi\to 0.
		\end{equation}
		Here $p_i : \Bl_\Delta (X \times X) \to X$ are the compositions of the blowup morphisms with the projections to the $i$th factor, the $\bZ/2$-action on the middle term permutes the two factors, $\iota: E \to \Bl_\Delta (X \times X)$ is the inclusion of the exceptional divisor and $\pi : E \to \Delta \simeq X$ is the obvious projection. $\chi$ denotes the nontrivial character of $\bZ/2$.
	\end{proposition}
	\begin{proof}
		Let $G = \bZ/2$. The sequence above can be viewed as a consequence of the natural semiorthogonal decomposition of the equivariant derived category $\Db_{G}(\cZ_2)$, which follows from \cite[Theorem 5.1]{gluing}: 
		\[\Db_{G}(\cZ_2)  = \langle \iota_*\Db(E) \otimes \chi, q^*\Db(X^{[2]})\rangle .\]
		Here, $\chi$ is the unique nontrivial character of $G = \bZ/2$. Note that the pushforward $\iota_* : \Db(E) \to \Db_{G}(\cZ_2)$ is taken with the trivial $G$-equivariant structure. As a consequence, there is a natural distinguished triangle of endofunctors on $\Db_{G}(\cZ_2)$ given by 
		\[q^*\left[(q_*(-))^G\right] \to \on{id} \to \iota_*((\iota^*(- \otimes \chi))_G) \otimes \chi \xrightarrow{\,+\,}\]
		where $(-)^G : \Db_G(X^{[2]}) \to \Db(X^{[2]})$ denotes taking $G$-invariants and $(-)_G : \Db_G(E) \to \Db(E)$ denotes taking $G$-coinvariants; these are respectively the left and right adjoints of the ``trivial representation functors'' $\Db(X^{[2]}) \to \Db_G(X^{[2]})$ and $\Db(E) \to \Db_G(E)$. \par 
		The proposition then follows from applying this distinguished triangle to the vector bundle $p_1^*\cE \oplus p_2^*\cE$, as 
		\[(q_*(p_1^*\cE \oplus p_2^*\cE))^G \simeq (\cE^{[2]} \oplus \cE^{[2]})^G \simeq \cE^{[2]},\]
		while 
		\[(\iota^*((p_1^*\cE \oplus p_2^*\cE) \otimes \chi))_G \simeq ((\pi^*\cE \oplus \pi^*\cE)\otimes\chi)_G \simeq \pi^*\cE.\]
	\end{proof}
	Note that the map $p_1^*\cE\oplus p_2^*\cE \to \iota_*\pi^*\cE \otimes \chi$ in the proposition is given by the map 
	\[(s_1,s_2)\mapsto s_1|_E - s_2|_E.\]
	\par 
	Let us take now a closed immersion $X \hookrightarrow Y$ of smooth varieties. It then induces a natural closed immersion $X^{[2]} \hookrightarrow Y^{[2]}$ by \cite[Tag 0B97]{stacks-project} which also arises as the $\bZ/2$-quotient of the natural map $\Bl_\Delta (X \times X) \to \Bl_\Delta (Y \times Y)$, and one can ask how the normal bundle $\cN_{X^{[2]}/Y^{[2]}}$ is controlled by the normal bundle $\cN_{X/Y}$. The answer is relatively simple: 
	\begin{proposition} \label{hilb2nmbdl}
		Suppose $X \subset Y$ is a smooth closed subvariety of a smooth variety. Then 
		\[\cN_{X^{[2]}/Y^{[2]}} \simeq \cN_{X/Y}^{[2]}.\]
	\end{proposition}
	\begin{proof}
		Observe that we have the pullback square 
		\begin{center}
			\begin{tikzcd}
				\Bl_\Delta (X \times X) \arrow[r,hook] \arrow[d,"q"] & \Bl_\Delta (Y \times Y) \arrow[d] \\
				X^{[2]} \arrow[r,hook] & Y^{[2]}
			\end{tikzcd}
		\end{center}
		with vertical maps double coverings and horizontal maps closed immersions; this diagram is evidently $\bZ/2$-equivariant with respect to the fiberwise involution of the double coverings and the trivial actions on $X^{[2]}$ and $Y^{[2]}$. Since the pullback of the defining ideal of $X^{[2]}$ in $Y^{[2]}$ generates the ideal for $\Bl_\Delta (X \times X)$ in $\Bl_\Delta (Y \times Y)$, we have a natural $\bZ/2$-equivariant surjection of vector bundles 
		\[q^*\cN_{X^{[2]}/Y^{[2]}}^\vee \to \cN^\vee_{\Bl_\Delta (X \times X)/\Bl_\Delta (Y \times Y)},\]
		and as these vector bundles are of the same rank we have a $\bZ/2$-equivariant equivalence of normal bundles
		\[q^*\cN_{X^{[2]}/Y^{[2]}} \simeq \cN_{\Bl_\Delta (X \times X)/\Bl_\Delta (Y \times Y)}.\]
		Now let $p_1, p_2 : \Bl_\Delta (X \times X) \to X$ denote the composition of the blowup morphisms and the first and second projections, let $\iota : E_X \to \Bl_\Delta (X \times X)$ be the inclusion of the exceptional divisor, and let $\pi : E_X \simeq \bbP(T_X) \to X$ denote the natural projection. By abuse of notation, we use the same notation for the analogous maps for $Y$. \par 
		Observe also that we have the natural map of $\bZ/2$-equivariant short exact sequences 
		\begin{center}
			\begin{tikzcd}
				0 \arrow[r]  & q^*T_{X}^{[2]} \arrow[r] \arrow[d] & p_1^*T_X \oplus p_2^*T_X \arrow[r] \arrow[d,equal] &  \iota_*\pi^*T_X \otimes \chi \arrow[r] \arrow[d]& 0 \\
				0 \arrow[r] & T_{\Bl_\Delta (X \times X)} \arrow[r] & p_1^*T_X \oplus p_2^*T_X \arrow[r] & \iota_*T_{E/X}(-1) \otimes \chi \arrow[r] & 0
			\end{tikzcd}
		\end{center}
		where the top sequence follows from Proposition \ref{sequenceprop} and the bottom one follows from the general exact sequence on tangent bundles for a blowup of a smooth variety along a smooth subvariety. The rightmost vertical map here is the pushforward by $\iota$ of the tautological one arising from the Euler sequence for the projective bundle, twisted by $\chi$: 
		\[0 \to \cO(-1) \to \pi^*T_X \to T_{E/X}(-1) \to 0.\]
		Hence the snake lemma yields an induces short exact sequence by examining the left column: 
		\[0 \to q^*T_X^{[2]} \to T_{\Bl_\Delta (X \times X)} \to \iota_*\cO(-1) \otimes \chi \to 0.\]
		Repeating this argument for $Y$ allows us to extend this to the map of $\bZ/2$-equivariant short exact sequences 
		\begin{center}
			\begin{tikzcd}
				0 \arrow[r] & q^*T_X^{[2]}\arrow[r] \arrow[d] & T_{\Bl_\Delta (X \times X)} \arrow[r] \arrow[d] & \cO(-1) \otimes \chi \arrow[r] \arrow[d] & 0 \\
				0 \arrow[r] & q^*T_Y^{[2]}|_{\Bl_\Delta (X \times X)} \arrow[r] & T_{\Bl_\Delta (Y \times Y)}|_{\Bl_\Delta (X \times X)} \arrow[r] & \cO(-1) \otimes \chi \arrow[r] & 0,
			\end{tikzcd}
		\end{center}
		where the right arrow is an isomorphism; hence the cokernels of the left two maps are isomorphic, but by the exactness of $(-)^{[2]}$ this yields a $\bZ/2$-equivariant equivalence 
		\[q^*\cN_{X/Y}^{[2]} \simeq q^*\cN_{X^{[2]}/Y^{[2]}}.\]
		Since $q^*$ is a fully faithful embedding into the equivariant derived category of $\Bl_\Delta (X \times X)$, the result follows. 
	\end{proof}
	\subsection{Standard flips and birational geometry} 
	For this section, suppose that $X$ is a smooth projective variety admitting a subvariety $\iota : Z \hookrightarrow X$ which admits a morphism $\pi : Z \to F$ that realizes $Z \simeq \bbP_F(\cE)$ as the projectivization of a rank $r+1$ vector bundle $\cE$ on $F$. Suppose moreover that the normal bundle $\cN_{Z/X}$ is of the form 
	\[\cN_{Z/X} \simeq \cO_\pi(-1) \otimes \pi^*\cE'\]
	for $\cE'$ a rank $s+1$ vector bundle on $F$. \par 
	Let us write \[\phi : \tilde{X}\to X\] for the blow up of $X$ in $Z$. The exceptional divisor $E \subset \tilde{X}$ of the blowup then fits into the following diagram of projective bundles: 
	\begin{center}
		\begin{tikzcd}
			& E \arrow[rd] \arrow[ld] & \\
			Z \simeq \bbP_F(\cE) \arrow[rd, "\pi"']  & & Z' \simeq \bbP_F(\cE') \arrow[ld, "\pi'"] \\
			& F & 
		\end{tikzcd}
	\end{center}
	where $E \simeq \bbP_Z(\pi^*\cE') \simeq \bbP_{Z'}(\pi^*\cE)$ admits the structure of a projective bundle over both $Z$ and $Z'$, so that over a point $F$, $E$ admits the structure of a $\bbP^r \times \bbP^s$-bundle. It follows from a calculation that the normal bundle $\cN_{E/\tilde{X}} \simeq \cO_E(E)$ is given by $\cO(-1,-1)$. The idea is then to contract the divisor $E \subset \tilde{X}$ in the ``opposite direction'', to produce a space $X'$ containing $Z'$ which is isomorphic in codimension 1 to $X$. \par 
	The construction of the blowdown $\phi' : \tilde{X} \to X'$ will follow from the following contraction theorem. 
	\begin{theorem}[Fujiki--Nakano contraction theorem, \cite{monoidalI,monoidalII}]
		Let $\tilde{Y}$ be a smooth complex variety of dimension $\geq 3$ containing a divisor $D \subset \tilde{Y}$ such that we have a morphism $D\to M$ realizing $D$ as a $\bbP^s$-bundle. If for each $m \in M$ the restriction of $\cO(D)$ to the fiber $D_m \simeq \bbP^s$ is isomorphic to $\cO(-1)$ then there exists a smooth algebraic space $Y$ containing $M$ such that $\tilde{Y} \to Y$ is a blowdown contracting $D$ to $M \subset Y$. 
	\end{theorem}
	Indeed, applying the above result to $\tilde{X}$ using that $E \to Z'$ is a projective bundle satisfying the necessary condition on the restriction of $\cO(E)$ shows that there exists a smooth and proper algebraic space $X' \supset Z'$ whose blowup in $Z'$ is $\tilde{X}$. \par 
	\begin{defn}
		For any smooth projective variety $X$ with a closed subvariety $Z \subset X$ which admits the structure of a projective bundle $\pi : Z \simeq \bbP(\cE)\to F$ over some other variety $F$ such we have an isomorphism $\cN_{Z/X} \simeq \cO_\pi(-1) \otimes \pi^*\cE'$, we call the birational map $X \longdashleftrightarrow X'$ (as well as the resulting diagram between $X, X', \Bl_{Z} X \simeq \Bl_{Z'} X'$ and their closed subspaces $Z, Z'$ and $E$) a \textbf{standard flip}. If $r + 1 = \on{rank} \cE = \on{rank} \cE' = s + 1$ then we say it is a \textbf{standard flop}. 
	\end{defn}
	Under nice situations, the arguments of \cite{projectiveflop} show that the contraction $X'$ as constructed above is a projective variety rather than just an algebraic space. Note that \cite{projectiveflop} was phrased in the case of flops; for the sake of completeness, we include their arguments here in the context of flips. However we expect that the statement in this generality is well-known to the experts. 
	\begin{proposition}[Proposition 1.3 of \cite{projectiveflop}]
		\label{contractionprop}
		Take $X$ as above in this section. Suppose moreover that there exists a birational map $g : X \to \bar{X}$ to a projective variety $\bar{X}$ containing $F$ which contracts $Z \subset X$ to $F \subset \bar{X}$. Then in fact the contraction $X'$ of $\tilde{X}$ above exists as a smooth projective variety. 
	\end{proposition}
	\begin{proof}
		The key idea is to show that the lines within the fibers of the projective bundle $E \to Z'$ (which are all in the same numerical class on $\tilde{X}$) are all $K_{\tilde{X}}$-negative extremal rays in $\bar{NE}(\tilde{X})$; by Mori's contraction theorem as in \cite{kollarmori}, it then follows that the contraction $X'$ of this curve class is a normal projective variety. The smoothness of $X'$ then follows from the Fujiki--Nakano contraction above. \par 
		We first check $K_{\tilde{X}}$-negativity. First suppose that $\tilde{\ell} \subset E \subset \tilde{X}$ is a line contracted by the map $E \to Z'$. Let $\ell \subset Z \subset X$ denote its image via the blow-up morphism, and observe that $\ell$ is contracted to a point by $\pi : Z \to F$. Now by the conormal sequence on $X$, one has 
		\[c_1(\cN_{Z/X}) = K_Z - K_X|_Z,\]
		so 
		\[(K_X.\ell) = (K_Z.\ell) - (c_1(\cN_{Z/X}).\ell) = (-r+1)-(-s+1) = s-r,\]
		using that on the fibers of $\pi$, $K_Z$ restricts to $\cO(-r+1)$ and $\cN_{Z/X}$ restricts to $\cO(-1)^{s+1}$. Hence by the projection formula and the identity $K_{\tilde{X}} = \phi^*K_X + sE$,
		\[(K_{\tilde{X}}.\tilde{\ell}) = (K_X.\ell) + s(E.\tilde{\ell}) = s - r - s = -r < 0.\]
		We now check that the class of $\tilde{\ell}$ is extremal in $\bar{NE}(\tilde{X})$. Here we apply the existence of the contraction map $g : X \to \bar{X}$, where $\bar{X}$ is projective. Suppose that $D$ is an ample line bundle on $\bar{X}$, and observe that $g^*D$ is a divisor supporting the curve $\ell$ on $X$. Suppose $H$ is ample on $X$, and let $c = (H.\tilde{\ell})$. Then for $k \gg 0$, the divisor 
		\[k\phi^*g^*D - (\phi^*H+cE)\]
		is big and nef and supports the curve $\tilde{\ell}$. It follows that the curve class corresponding to $\tilde{\ell}$ is extremal: if $[c_1] + [c_2] = [\tilde{\ell}]$ then $[c_1], [c_2]$ intersect trivially with the divisor supporting $\tilde{\ell}$ above, so they must lie in the ray of $\bar{NE}(\tilde{X})$ spanned by $[\tilde{\ell}]$. 
	\end{proof}
	One nice feature of these standard flips is that they induce decompositions of the invariants of $X$ in terms of $X'$, or vice versa. In particular, they induce natural semiorthogonal decompositions of derived categories. 
	\begin{theorem}[\!\!\cite{bondalorlov},\cite{cubicflip}]
		Suppose that $s \leq r$. Then we have a semiorthogonal decomposition 
		\begin{equation}
			\Db(X) = \langle \Phi_{-r+s}(\Db(F)), \cdots, \Phi_{-1}(\Db(F)), R\phi_* \circ L\phi'^*\, \Db(X')\rangle,
		\end{equation}
		where the functors $\Phi_m = \iota_*(\pi^*(-) \otimes \cO_\pi(m)): \Db(F) \to \Db(X)$ and $R\phi_* \circ L\phi'^* : \Db(X') \to \Db(X)$ are all fully faithful. 
	\end{theorem}
	This should be interpreted as a strengthening of one case of the more general \textit{DK conjecture} of Bondal--Orlov \cite{bondalorlov} and Kawamata \cite{dkconj} stating that any flip, not just a standard flip, should give rise to a fully faithful embedding of derived categories.
	\begin{remark}
		Though we choose not to focus on this here, one can check easily that standard flips also induce similar decompositions on the level of the Grothendieck ring of varieties. It has also been shown that they induce decompositions on the level of integral Chow motives \cite{chowmotives}.
	\end{remark}
	\section{Flips via the Nakano contraction theorem} \label{contractionflipssection}
	Let $(X,H)$ be a del Pezzo variety of dimension $n\geq 2$ such that $\cO(H)$ is very ample. We first require a more basic lemma on the normal bundle for any line in $X$. 
	\begin{proposition}
		\label{splittingprop1}
		For any line $\ell \subset X \subset \bbP(H^0(\cO(H))^\vee) \simeq \bbP^N$, we have either 
		\[\mathcal{N}_{\ell/X} \simeq \begin{cases}
			\mathcal{O}^{\oplus 2} \oplus \mathcal{O}(1)^{\oplus(n-3)}\\ 
			\mathcal{O}(-1) \oplus \mathcal{O}(1)^{\oplus(n-2)}.
		\end{cases}\]
		If $n = 2$, then only the second case occurs. 
	\end{proposition}
	\begin{proof}
		We have a canonical exact sequence 
		\[0 \to \cN_{\ell/X} \to \cN_{\ell/\bbP^N} \to \cN_{X/\bbP^N}|_\ell \to 0,\]
		but one observes that, $\ell$ is a line in $\bbP^N$ we have $\mathcal{N}_{\ell/\bbP^N} \simeq \mathcal{O}(1)^{\oplus(N-1)}$. Since $\cN_{\ell/X}$ is a vector bundle on $\bbP^1$, it splits, so we may write $\cN_{\ell/X} \simeq \bigoplus_{i=1}^{n-1}\cO_\ell(a_i)$, and the exact sequence above implies that $a_i \leq 1$. \par 
		Now since the inclusion $\ell \subset X$ is a local complete intersection, it follows that 
		\begin{align}
			\det \cN_{\ell/X} &\simeq \omega_{\ell} \otimes \omega_X^\vee \nonumber\\
			&\simeq \cO_\ell(-2) \otimes \cO_X((n-1)H)|_\ell \\
			&\simeq \cO_\ell(n-3). \nonumber
		\end{align}
		It follows that $\sum_{i=1}^{n-1} a_i = n-3$, with $a_i \leq 1$. The proposition follows immediately from these two conditions. 
	\end{proof}
	\begin{corollary}
		Under the hypotheses of this section, the Fano scheme of lines $F_1(X)$ is smooth of dimension $2n-4$.
	\end{corollary}
	\begin{proof}
		This is an immediate consequence of the vanishing of $H^1(\cN_{\ell/X}) = 0$ along with general results on tangent spaces to Hilbert schemes along with an easy calculation that $\dim H^0(\cN_{\ell/X}) = 2n-4$. 
	\end{proof}
	Under the conditions of Proposition \ref{splittingprop1}, it is then easy to check that this induces certain splittings on the normal bundle of $\ell^{[2]}$ in $X^{[2]}$. This is a consequence of splittings for tautological bundles on $\ell^{[2]} \simeq \bbP^2$. 
	\begin{proposition}
		We have the following splittings on $\ell^{[2]} = \bbP^2$: 
		\begin{itemize} \label{p2splittinglemma}
			\item $\cO(-1)^{[2]} \simeq \cO(-1) \oplus \cO(-1)$,
			\item $\cO^{[2]} \simeq \cO \oplus \cO(-1)$, and
			\item $\cO(1)^{[2]} \simeq \cO \oplus \cO$. 
		\end{itemize}
	\end{proposition}
	\begin{proof}
		Let $q : \bbP^1 \times \bbP^1 \to (\bbP^1)^{[2]} \simeq \bbP^2$ and $p : \bbP^1 \times \bbP^1 \to \bbP^1$ projection onto the first factor. By the Horrocks splitting criterion on $\bbP^2$, to check that these rank 2 vector bundles split it suffices to check that every twist has no intermediate cohomology. Note that $q^*\cO(1) \simeq \cO(1,1)$. Moreover, observe that the projection formula yields 
		\[\cO(d)^{[2]} \otimes \cO(n) \simeq q_*p^*\cO(d) \otimes \cO(n) \simeq q_*\cO(n+d,n).\]
		Hence $R\Gamma(\cO(d)^{[2]} \otimes \cO(n)) \simeq R\Gamma(\cO(n+d,n))$ since $q$ is finite. It is therefore an immediate consequence of the K\"unneth formula for sheaf cohomology that if $\abs{d} \leq 1$, $H^1(\cO(d)^{[2]}\otimes\cO(n)) = 0$ for all $n \in \bZ$, and the vector bundles $\cO(d)^{[2]}$ split in this range. \par 
		To see which line bundles these rank 2 vector bundles split into requires some easy casework: 
		\begin{itemize}
			\item $\cO(-1)^{[2]}$ and its twist  $\cO(-1)^{[2]} \otimes \cO(-1)$ are acyclic on $\bbP^2$, so all of its summands must be isomorphic to $\cO(-1)$; hence $\cO(-1)^{[2]} \simeq \cO(-1)\oplus \cO(-1)$.
			\item The splitting $\cO^{[2]} = q_*\cO \simeq \cO \oplus \cO(-1)$ follows from the fact that $q$ is a branched covering branched over a conic in $\bbP^2$, along with a general result for the pushforward of the structure sheaf for branched coverings. 
			\item For $\cO(1)^{[2]}$, observe that the global sections are 2-dimensional; but the only split rank 2 vector bundle on $\bbP^2$ with 2-dimensional global sections is $\cO \oplus \cO$. 
		\end{itemize} 
	\end{proof}
	We can now extract information about the normal bundle for the inclusion $\ell^{[2]} \hookrightarrow X^{[2]}$. Putting together the above results along with the additivity of the functor $\cE \mapsto \cE^{[2]}$ immediately implies the following:
	\begin{proposition}
		There is an isomorphism on $\mathbb{\ell}^{[2]} \simeq \mathbb{P}^2$: 
		\[\mathcal{N}_{\ell^{[2]}/X^{[2]}} \simeq \mathcal{O}(-1)^{\oplus 2} \oplus \mathcal{O}^{\oplus(2n-4)}.\]
	\end{proposition}
	\begin{proof}
		Combine Propositions \ref{hilb2nmbdl}, \ref{splittingprop1} and \ref{p2splittinglemma}. 
	\end{proof}
	Using this, it is not too difficult to finish the proof of Theorem \ref{delpezzoflip}. 
	\begin{proof}[Proof of Theorem \ref{delpezzoflip}]
		We would like to invoke the Nakano contraction theorem to the blowup of $X^{[2]}$ along the subvariety $\bbP(\Sym^2 \cU_2^\vee)$, and by the discussion above it suffices to calculate the restricted normal bundle $\cN_{\bbP(\Sym^2 \cU_2)/X^{[2]}}$. However, one can observe that $\bbP(\Sym^2 \cU_2^\vee)$ coincides with the relative Hilbert scheme of 2 points for the universal line $\bbL \to F_1(X)$. To make this explicit, we write $\bbL^{[2/F_1(X)]} := \bbP(\Sym^2\cU_2^\vee)$. \par 
		From the sequence of inclusions $\ell^{[2]} \subset \mathbb{L}^{[2]/F_1(X)} \subset X^{[2]}$, we get a sequence of normal bundles 
		\[0 \to \mathcal{N}_{\ell^{[2]}/\mathbb{L}^{[2]/F_1(X)}} \to \mathcal{N}_{\ell^{[2]}/X^{[2]}} \to \mathcal{N}_{\mathbb{L}^{[2]/F_1(X)}/X^{[2]}}|_{\ell^{[2]}} \to 0.\]
		However, the first term is necessarily trivial since $\ell^{[2]}$ is simply the fiber of the projective bundle morphism $\mathbb{L}^{[2]/F_1(X)} \to F_1(X)$, so this can be rewritten as 
		\[0 \to \mathcal{O}^{\oplus(2n-4)} \to \mathcal{O}(-1)^{\oplus2}\oplus \mathcal{O}^{\oplus(2n-4)} \to \mathcal{N}_{\mathbb{L}^{[2/F_1]}/X^{[2]}}|_{\ell^{[2]}} \to 0,\]
		so the last term must be $\mathcal{O}(-1)^{\oplus 2}$. In particular, $\cN_{\bbL^{[2/F_1]}/X^{[2]}} \simeq \pi^*\cE \otimes \cO(-1)$ for some vector bundle $\cE$ on $F_1(X)$. \par 
		In order to finish the theorem, we wish to apply Proposition \ref{contractionprop} to the projective variety $X^{[2]}$ with role of the subvariety being played by $\bbL^{[2/F_1]} \subset X^{[2]}$. For the required map from $X^{[2]}$ to a projective variety contracting $L^{[2/F_1]}$, observe that there is a natural morphism 
		\[X^{[2]} \to \Gr(2,H^0(\cO(H))^\vee)\] 
		which takes a length two subscheme of $X^{[2]} \subset \bbP(H^0(\cO(H))^\vee)$ to the line it spans. Taking the Stein factorization yields a map $X^{[2]} \to \bar{Y} \to \Gr(2,H^0(\cO(H))^\vee)$ which is an isomorphism onto its image away from $\bbL^{[2/F_1]}$ and contracts this subvariety to $F_1(X) \subset \Gr(2,H^0(\cO(H))^\vee)$. Hence Proposition \ref{contractionprop} applies and we get the standard flip diagram \eqref{flipdiag}. \par 
		The existence of $\Db(F_1(X))$ as a semiorthogonal component then follows immediately from Theorem A of \cite{cubicflip}. 
	\end{proof}
	\section{Flips via linear spans of quadrics} \label{geomflipsection}
	In each part of Theorem \ref{geomflipthm}, the flavor of the construction of the putative flip between the Hilbert scheme of quadrics $G_k(X)$ and its birational partner is the same: given a quadric $\Sigma \in G_k(X)$ we take the $(k+1)$-plane $\langle \Sigma \rangle$ spanned by $\Sigma$ in the ambient projective space before intersecting with $X$ again. In each case, we find ourselves in the following highly special situation. 
	\begin{situation} \label{flippingsituation}
		Assume that we have the following data: 
		\begin{itemize}
			\item a proper variety $G$ with vector bundles $\cE$ and $\cE'$,
			\item a pair of closed subvarieties $Y \hookrightarrow \bbP(\cE)$ and $Y' \hookrightarrow \bbP(\cE')$, with $Y'$ smooth and irreducible, and
			\item a smooth, possibly disconnected closed subvariety $F \hookrightarrow G$ such $Y \times_G F \to \bbP(\cE|_F)$ is an isomorphism and $Y' \times_G F \to \bbP(\cE'|_F)$ is a (possibly empty) equidimensional union of projective subbundles over the connected components of $F$. 
		\end{itemize}
		We summarize this in the following (rather large) diagram: 
		\begin{equation}
			\begin{tikzcd}[sep=small]
			Y \times_G F \arrow[r,"\sim"] & \bbP(\cE|_F) \arrow[r,hook]\arrow[ddrrr]& Y  \arrow[r,hook]& \bbP(\cE) \arrow[rd] & & \bbP(\cE') \arrow[ld] & Y \arrow[l,hook'] & \bbP(\cE'|_F) \arrow[l,hook'] \arrow[ddlll]& Y' \times_G F \arrow[l,hook']\\
			& & & & G & & & & \\
			& & & & F \arrow[u,hook]& & & &
			\end{tikzcd}
		\end{equation}
		Given this data, we may define $Z := Y \times_G F$, $Z' := Y' \times_G F$, along with their open complements $U := Y \setminus Z$, $U' := Y' \setminus Z'$. \par 
		Write $\cO_{Y'}(-1)$ for the restriction of the tautological subbundle from the projective bundle $\bbP(\cE')$. Now suppose that $Y$ is smooth along $Z$, and that on $Y'$ there exists a map $\textbf{s} : \cO_{Y'}(1) \rightarrow \cE$, such that: 
		\begin{itemize}
			\item $\textbf{s}$ is injective with locally free cokernel along $U'$,
			\item the induced map $U' \to \bbP(\cE)$ is an isomorphism onto $U$.
		\end{itemize}
		It then follows immediately that $Y$ is smooth. 
	\end{situation}
	For the cases of our interest, we will always take $Y = G_k(X)$ to be the Hilbert scheme of quadrics for $X \subset \bbP(V)$ our del Pezzo variety, $G = \Gr(k+2,V)$ to be the Grassmannian of projective $(k+1)$-planes, $F = F_{k+1}(X)$ to be the Fano variety of $(k+1)$-planes contained within $X$, and $\cE = \Sym^2 \cU_{k+2}^\vee$ where $\cU_{k+2}$ is the tautological subbundle on $\Gr(k+2,V)$. 
	\begin{warning}
		As we will frequently need to pull back the bundle $\cU_{k+2}$ on $G = \Gr(k+2,V)$ to the varieties $F, Y$ and $Y'$, we often abuse notation by suppressing pullbacks and restrictions and writing simply $\cU_{k+2}$; generally it will be clear from context which variety the bundle lives on. 
	\end{warning}
	\begin{lemma}\label{flippinglemma}
		Suppose that we are in the setting of Situation \ref{flippingsituation}.
		\begin{enumerate}
			\item If $Z' = \varnothing$, then $Y = Z \sqcup U$. In particular if $F = \varnothing$, then $Y = U$. 
			\item If $Z'$ is nonempty, then $Y$ is irreducible. If moreover $Z' = V(\textbf{s})$ has codimension equal to $\on{rank}\cE$ (so that in the language of \cite{cbs}, $Z'$ is a complete bundle-section), then $\Bl_Z Y \simeq \Bl_{Z'} Y'$ and the diagram 
			\begin{center}
				\begin{tikzcd}
					Z \arrow[r, hook] \arrow[d] & Y \arrow[r, dashed, leftrightarrow] & Y' & Z' \arrow[l, hook'] \arrow[d] \\
					F & & & F 
				\end{tikzcd}
			\end{center}
			is a standard flip. 
		\end{enumerate}
	\end{lemma}
	\begin{proof}
		(1) is clear: by assumption $Z$ is smooth, and the isomorphism $U \simeq U'$ ensures that $U$ is smooth, irreducible and proper; then the fact that $U$ is proper ensures that it is closed in $Y$, decomposing $Y$ into connected components. \par 
		For the claims of (2): we first show the irreducibility of $Y$, which is straightforward. Since $Y$ is smooth, it suffices to show that $Y$ is connected. By assumption, $Z'$ is an equidimensional union of projective bundles over the components of $F$, so in particular $Z' \to F$ is surjective and $F$ lies inside the image
		\[\im(Y' \to G) = \bar{\im(U' \to G)} = \bar{\im(U \to G)} = \im(\bar{U} \to G).\] 
		In particular, the irreducible subset $\bar{U}$ must meet each connected component of $Z$, and as $U \cup Z = Y$ we conclude that $Y$ is connected. 
		\par 
		Assume now that $Z' = V(\textbf{s})$ has codimension equal to $\on{rank} \cE$. Then there is a natural embedding of $\Bl_{Z'} Y' \hookrightarrow \bbP(\cE) \times_G Y'$ whose restriction to $U'$ is the morphism induced by the section $\textbf{s}$. Composing with the projection $\bbP(\cE) \times_G Y' \to \bbP(\cE)$ therefore yields a morphism $\Bl_{Z'} Y' \to \bbP(\cE)$ whose restriction to $U'$ gives the isomorphism onto $U \subset Y$. It follows from taking closures that the image of the morphism $\Bl_{Z'} Y' \to \bbP(\cE)$ lands in $Y$. \par 
		However, since $Z$ and $Z'$ are both the pullback of $F$ from $G$, the pullback of $Z$ along this morphism is precisely the exceptional divisor of the blowup; in particular, we find a birational map $\Bl_{Z'} Y' \to \Bl_Z Y$ sending exceptional divisor to exceptional divisor. Since this birational map cannot contract any divisor, purity of the exceptional locus ensures that this map is an isomorphism. \par 
		To see that this is a standard flip, first observe that the exceptional divisor of $\Bl_Z Y = \Bl_{Z'}Y'$ can be identified with the biprojective bundle $Z \times_F Z'$ with the maps to $Z$ and $Z'$ being the two natural contractions. For this to be a standard flip, it suffices to compute the normal bundle $\cN_{Z'/Y'}$; but since $Z'$ is cut out by the regular section of the bundle $\cO_{Y'}(-1) \otimes \cE$ induced by $\textbf{s}$, a standard computation using Koszul resolutions shows that in fact $\cN_{Z'/Y'} \simeq \cO_{Z'}(-1) \otimes \cE $. 
	\end{proof}
	\begin{remark}
		Note that in the case when $F$ is disconnected, we may instead view this as the composition of the standard flips along the connected components of $Z'$. 
	\end{remark}
	The meat of our analyses will be showing that the following situations all fall into the regime above. 
	\subsection{The case of cubics}
	We now work in the setting of Theorem \ref{geomflipthm}(a); in particular, we let $X \subset \bbP^{n+1} = \bbP(V)$ be a smooth $n$-dimensional cubic hypersurface such that $F_k(X)$, $F_{k+1}(X)$ are smooth and irreducible of the expected dimensions 
	\[\dim F_k(X) = (k+1)(n+1-k)-{k+3\choose 3} \text{ and } \dim F_{k+1}(X) = (k+2)(n-k) - {k + 4 \choose 3}.\]
	Note that this is true for a generic cubic hypersurface by \cite[Th\'eor\`eme 2.1b)]{debarremanivel}.
	Our precise version of the theorem for cubics is as follows: 
	\begin{theorem} \label{geomflipcubics}
			Fix an integer $0 \leq k \leq n$. Suppose $X \subset \bbP^{n+1}$ is a smooth cubic hypersurface such that $F_k(X)$ and $F_{k+1}(X)$ are smooth and irreducible of the expected dimension. Let $\cU_{k+2}$ denote the tautological subbundle of rank $k+2$ on $\Gr(k+2,V)$ and $\cQ_{k+1}$ denote the tautological quotient bundle of rank $(n+2) - (k+1)$ on $F_k(X)$. Then we have a standard flip diagram 
		\begin{equation}
			\begin{tikzcd}
				\bbP(\Sym^2 \cU_{k+2}^\vee|_{F_{k+1}(X)}) \arrow[d] \arrow[r, hook]& G_k(X) \arrow[leftrightarrow, r, dashed, "\text{flip}"] & \bbP(\cQ_{k+1}) \arrow[d] & \arrow[l, hook'] F_{k,k+1}(X) \arrow[d]\\
				F_{k+1}(X) & & F_k(X) & F_{k+1}(X)
			\end{tikzcd}
		\end{equation}
		In particular, $G_k(X)$ is smooth and irreducible of the expected dimension. 
	\end{theorem}
	\begin{remark}
		In particular, if $n < k + \frac{1}{6}(k+3)(k+2) - 1$ so that $F_k(X) = \varnothing$, $G_k(X) = \varnothing$. Otherwise if $ k + \frac{1}{6}(k+4)(k+3) > n \geq k + \frac{1}{6}(k+3)(k+2) - 1$ so that $F_{k+1}(X) = \varnothing$, the flip becomes an isomorphism $G_k(X) \simeq \bbP_{F_k(X)}(\cQ_{k+1})$. \par 
		Also, while we have assumed for simplicity that the Fano varieties $F_k(X)$ and $F_{k+1}(X)$ are irreducible, this assumption need not always be necessary. Indeed, by \cite[Theorem 4.1]{borcea} as long as a Fano variety of $k$-planes on a cubic is of the expected dimension and is positive dimensional, it is necessarily connected (and by our assumption of smoothness necessarily must be irreducible), so the assumption of irreducibility is often extraneous; moreover, to be in Situation \ref{flippingsituation} it is permissible for $F_{k+1}(X)$ to be disconnected as long as $F_k(X)$ is connected. 
	\end{remark}
	The variety $\bbP(\cQ_{k+1})$ can be described as the subvariety of $\Fl(k+1,k+2,V)$ given by pairs of the form $\{(U_{k+1} \subset U_{k+2}) : \bbP(U_{k+1}) \subset X\}.$ \par 
	In order to prove the theorem we first show that we are in Situation \ref{flippingsituation} so that we can apply Lemma \ref{flippinglemma}. As suggested before, we take: 
	\begin{itemize}
		\item $G = \Gr(k+2,V)$, with vector bundles $\cE = \Sym^2 \cU_{k+2}^\vee$ and $\cE' = \cU_{k+2}^\vee$, 
		\item $Y = G_k(X)$ and $Y' = \bbP(\cQ_{k+1})$, with natural inclusions $G_k(X) \hookrightarrow G_k(\bbP^n) \simeq \bbP_G(\Sym^2\cU_{k+2}^\vee)$ and $\bbP(\cQ_{k+1}) \hookrightarrow \Fl(k+1,k+2,V) \simeq \bbP(\cU_{k+2}^\vee)$, and 
		\item $F = F_{k+1}(X)$, with $Z := Y \times_G F \simeq \bbP(\Sym^2 \cU_{k+2}^\vee|_F)$ consisting of quadrics on $X$ whose linear span also lies in $X$, and $Z' := Y' \times_G F \simeq \bbP( \cU_{k+2}^\vee|_F) = F_{k,k+1}(X)$ consists of flags in $\Fl(k+1,k+2,V)$ whose projectivizations lie in $X$. 
	\end{itemize}
	To check that $G_k(X)$ is smooth along $Z$, we use deformation theory. 
	\begin{lemma} \label{smoothnesslemmacubics}
		The variety $G_k(X)$ is smooth along the subvariety $Z$. 
	\end{lemma}
	\begin{proof}
		For each $k$-dimensional quadric $\Sigma \in Z$, the $(k+1)$-plane $\langle \Sigma \rangle$ spanned by the quadric must lie within the cubic. In particular, the inclusion $\Sigma \subset \langle \Sigma \rangle$ and the inclusion $\langle \Sigma \rangle \subset X$ are both local complete intersections, so $\Sigma \subset X$ is also a local complete intersection. Hence by deformation theory, it suffices to check that $h^1(\cN_{\Sigma/X}) = 0$. \par 
		To calculate this, since each map is a local complete intersection, we have the two normal bundle sequences 
		\begin{equation}
			\begin{gathered}
				0 \to \cN_{\Sigma/\langle \Sigma \rangle} \simeq \cO_\Sigma(2) \to \cN_{\Sigma/X} \to \cN_{\langle \Sigma \rangle/X}|_\Sigma \to 0 \\
				0 \to \cN_{\langle \Sigma \rangle/X} \to \cN_{\langle \Sigma \rangle/\bbP^{n+1}} \simeq \cO_{\langle\Sigma\rangle}(1)^{\oplus(n-k)}  \to \cN_{X/\bbP^{n+1}}|_{\Span{\Sigma}} \simeq \cO_{\Span{\Sigma}}(3) \to 0. \label{nmbdlsq1}
			\end{gathered}
		\end{equation}
		Since the Fano variety of $(k+1)$-planes on $X$ is smooth of the expected dimension, it follows that \[\dim T_{F_{k+1}(X),\Span{\Sigma}} = h^0(\cN_{\Span{\Sigma}/X}) = (k+2)(n-k) - {k+4 \choose 3},\]
		and an easy dimension calculation shows that the second map in the induced sequence on $H^0$ from the second sequence in \eqref{nmbdlsq1}, given by 
		\[0 \to H^0(\cN_{\Span{\Sigma}/X}) \to H^0(\cO_{\Span{\Sigma}}(1))^{\oplus(n-k)} \to H^0(\cO_{\Span{\Sigma}}(3)),\]
		must necessarily be surjective, forcing $h^1(\cN_{\Span{\Sigma}/X}) = 0$. Twisting the second sequence of \eqref{nmbdlsq1} by $(-2)$ and taking cohomology, we find the exact sequence
		\[0 = H^1(\cO_{\Span{\Sigma}}(1)) \to H^2(\cN_{\Span{\Sigma}/X}(-2)) \to H^2(\cO_{\Span{\Sigma}}(-1))^{\oplus(n-k)} = 0,\]
		forcing $h^2(\cN_{\Span{\Sigma}/X}(-2)) = 0$. On the other hand, we have the restriction sequences 
		\begin{gather*}
			0 \to \cN_{\langle \Sigma \rangle/X}(-2) \to \cN_{\langle \Sigma \rangle/X} \to \cN_{\langle \Sigma \rangle/X}|_\Sigma \to 0 \\
			0 \to \cO_{\Span{\Sigma}} \to \cO_{\Span{\Sigma}}(2) \to \cO_\Sigma(2) \to 0.
		\end{gather*}
		Taking cohomology immediately implies that $h^1(\cO_\Sigma(2)) = h^1(\cN_{\Span{\Sigma}/X}|_\Sigma) = 0$, and hence using the first sequence of \eqref{nmbdlsq1} we find $h^1(\cN_{\Sigma/X}) = 0$. \par 
	\end{proof}
	In order to construct our map $\textbf{s} : \cO_{Y'}(1) \to \Sym^2 \cU_{k+2}^\vee$, we first observe that the biprojective bundle \[\bbP(\cU_{k+2}^\vee) \times_G \bbP(\cU_{k+2})\] carries a tautological $(1,1)$-divisor. If we pull this back to $Y' \times_G \bbP(\cU_{k+2})$ along the map $Y' = \bbP(\cQ_{k+1}) \hookrightarrow \bbP(\cU_{k+2}^\vee)$, one can observe that the defining $(0,3)$-divisor coming from the defining equation of the cubic $X$ vanishes along the tautological $(1,1)$-divisor. In particular, the defining equation $\cO \to \cO(0,3)$ factors as 
	\[\cO \to \cO(-1,2) \to \cO(0,3)\]
	where the second map is multiplication by the defining section of the tautological (1,1)-divisor. Pushing forward along the map $Y' \times_G \bbP(\cU_{k+2}) \to Y'$, we get a factorization of the defining equation for the cubic, 
	\begin{equation}
		\cO \to \cO_{Y'}(-1) \otimes \Sym^2\cU^\vee_{k+2} \to \Sym^3 \cU_{k+2}^\vee, \label{cubicsectiondef1}
	\end{equation}
	and we define $\textbf{s} : \cO_{Y'}(1) \to \Sym^2 \cU_{k+2}^\vee$ to be up to twisting given by the first map in this factorization. \par 
	Geometrically, given a flag $(\bbP(U_{k+1}), \bbP(U_{k+2})) \in Y'$, $\textbf{s}$ picks out the quadric on $\bbP(U_{k+2})$ given by dividing the restriction to $\bbP(U_{k+2})$ of the defining equation for $X$ by the equation for the hyperplane $\bbP(U_{k+1}) \subset \bbP(U_{k+2})$. 
	\begin{lemma} \label{sectionlemmacubics}
		The section $\textbf{s}$ is injective with locally free cokernel along $U'$ and induces an isomorphism onto $U$. Moreover, $Z' = V(\textbf{s})$ has codimension equal to $\on{rank} \Sym^2 \cU_{k+2}^\vee$. 
	\end{lemma}
	\begin{proof}
		Recall that we defined $U'$ as the complement to $Z' = Y' \times_G F$; in other words, it is precisely the locus on $Y' = \bbP(\cQ_{k+1})$ where the defining equations of $F = F_{k+1}(X)$ do not vanish. However, $F_{k+1}(X)$ is defined on $G = \Gr(k+2,V)$ precisely by the vanishing of the section 
		\[\cO\to \Sym^3 \cU_{k+2}^\vee\] 
		coming from the equation of the cubic, so on $U'$ this map is fiberwise nonvanishing; so the section $\textbf{s}$, which factors this section, is also fiberwise nonvanishing. It then follows that $\textbf{s}$ is injective with locally free cokernel along $U'$, and hence defines a map $U' \to \bbP(\Sym^2 \cU_{k+2}^\vee)$. By construction this family of quadrics lies in the cubic $X$, and so it is clear that the image lands inside $U \subset G_k(X)$. \par 
		To check that this induced map gives an isomorphism onto $U$, we first observe that there is a natural choice of inverse map $U \to \bbP(\cU_{k+2}^\vee)$ induced by a section $\textbf{s}'$. As before, one can observe that the bundle $\bbP(\Sym^2\cU_{k+2}^\vee) \times_G \bbP(\cU_{k+2})$ carries a tautological divisor, which in this case is of type $(2,1)$; after restricting to $Y = G_k(X)$, since the $(0,3)$-divisor coming from the defining equation of the cubic vanishes along the tautological $(2,1)$-divisor, we find a factorization 
		\[\cO\to \cO(-1,1) \to \cO(0,3),\]
		and by pushing forward find a factorization
		\begin{equation}
			\cO\to \cO_Y(-1) \otimes\cU_{k+2}^\vee \to \Sym^3 \cU_{k+2}^\vee, \label{cubicsectiondef2}
		\end{equation}
		and let $\textbf{s}' : \cO_Y(1) \to \cU_{k+2}^\vee$ be the twist of the first map. Once again, since the equation of the cubic vanishes along this family of planes, the image of the induced map $U \to \bbP(\cU_{k+2}^\vee)$ lands inside of $U'$. \par 
		To check that the composition $U \to U' \to U$ is the identity, it suffices to check that the pullback of $\textbf{s} : \cO_{Y'}(1) \to \Sym^2 \cU_{k+2}^\vee$ along the map $U \to U'$ induced by $\textbf{s}'$ is the tautological quadric on $U \subset G_k(X)$. But both the tautological quadric and the pullback of $\textbf{s}$ can be realized as the residual family of quadrics to the family of planes coming from the section $\textbf{s}'$ in the family of cubics $U \times_G \bbP(\cU_{k+2}) \times_{\bbP(V)} X$, so these two families of quadrics must coincide. The composition $U' \to U \to U'$ is similar, and we conclude that we have isomorphisms between $U$ and $U'$.  \par 
		For the last claim, we first verify the equality $V(\textbf{s}) = Z'$. Indeed, $Z'$ is cut out by the pullback to $Y'$ of the defining section $\cO \to \Sym^3 \cU_{k+2}^\vee$ of $F = F_{k+1}(X)$ in $G$ which by the factorization \eqref{cubicsectiondef1} factors as 
		\[\cO \xrightarrow{\textbf{s}} \cO_{Y'}(-1) \otimes \Sym^2 \cU_{k+2}^\vee \to \Sym^3 \cU_{k+2}^\vee,\]
		where the second map itself factors as $\cO_{Y'}(-1) \otimes \Sym^2 \cU_{k+2}^\vee \to \cU_{k+2}^\vee \otimes \Sym^2 \cU_{k+2}^\vee \to \Sym^3 \cU_{k+2}^\vee$, where the first map comes from the tautological section on $\bbP(\cU_{k+2}^\vee)$ and the second comes from multiplication. But since the map $\cO_{Y'}(-1) \otimes \Sym^2 \cU_{k+2}^\vee \to \Sym^3 \cU_{k+2}^\vee$ is easily checked fiberwise to be an injection with locally free cokernel, we conclude that the composite section $\cO\to \Sym^3 \cU_{k+2}^\vee$ vanishes if and only if $\textbf{s}$ vanishes. \par 
		We finally conclude that $\dim Y' - \dim Z' = \on{rank} \Sym^2 \cU_{k+2}^\vee$ by observing that we have an identity of dimensions \begin{equation}
			\begin{aligned}
			\dim Z' &= (k+2)(n-k) - {k+4 \choose 3}+ k + 1 \\
			&= (k+1)(n+1-k) - {k+3 \choose 3} + (n - k) - {k+3 \choose 2} \\
			&= \dim Y' - \on{rank}\left(\Sym^2 \cU_{k+2}^\vee\right).
			\end{aligned}
		\end{equation}
	\end{proof}
	At this point, the proof of the theorem is immediate. 
	\begin{proof}[Proof of Theorem \ref{geomflipcubics}]
		Observe that we are in Situation \ref{flippingsituation} and combine Lemmas \ref{flippinglemma}, \ref{smoothnesslemmacubics} and \ref{sectionlemmacubics}. The statement of the theorem follows. 
	\end{proof}
	\subsection{The case of the complete intersection of two quadrics}
	The analogous theorem we prove for complete intersections of two quadrics is as follows: 
	\begin{theorem}
		\label{geomflipquartic}
		Fix an integer $0 \leq k < n$. Suppose $X = Q_1 \cap Q_2 \subset \bbP^{n+2}$ is a smooth complete intersection of two quadric hypersurfaces. Let $\fQ$ denote the total space of the quadric fibration for the pencil of quadrics $\langle Q_1, Q_2 \rangle$, and write $\OGr(k+2,\fQ)$ for the relative orthogonal Grassmannian parametrizing isotropic $(k+2)$-dimensional subspaces of $V$ for the quadric fibration. Let $W \subset \Sym^2 V^\vee$ be the 2-dimensional linear subspace of the space of all quadrics corresponding to the pencil $\langle Q_1, Q_2 \rangle$. Let $\cU_{k+2}$ denote the tautological subbundle of rank $k+2$ on $\Gr(2,V_5)$. Then we have a standard flip diagram 
		\begin{equation}
			\begin{tikzcd}
				\bbP(\Sym^2 \cU_{k+2}^\vee|_{F_{k+1}(X)}) \arrow[d] \arrow[r, hook]& G_k(X) \arrow[leftrightarrow, r, dashed, "\text{flip}"] & \OGr(k+2,\fQ) \arrow[d] & \arrow[l, hook'] F_{k+1}(X) \times \bbP(W) \arrow[d]\\
				F_{k+1}(X) & &\langle Q_1, Q_2 \rangle = \bbP(W) & F_{k+1}(X)
			\end{tikzcd}
		\end{equation}
	\end{theorem}
	\begin{remark}
		If $n < k - 1 + \frac{1}{2}(k+3)$ so that the orthogonal Grassmannian fibration $\OGr(k+2,\fQ) = \varnothing$, $G_k(X)$ will be empty as well. Otherwise, if $k - 1 + (k+3) > n \geq k - 1 + \frac{1}{2}(k+3)$ so that $F_{k+1}(X) = \varnothing$, the flip becomes an isomorphism $G_k(X) \simeq \OGr(k+2,\fQ)$.  \par 
		This result gives an alternate proof for \cite[Proposition A.7]{colkuz}, which used the machinery of Springer resolutions to prove a large part of the same result above. 
	\end{remark}
	Due to a theorem of Reid \cite[Theorem 2.6]{reidthesis}, as long as $X$ is smooth it follows that $F_{k+1}(X)$ is also smooth of the expected dimension $(k+2)(n-k+1) - 2{k+3 \choose 2}$. On the other hand, by \cite[Lemma 3.6]{ogsmoothness}, the fibration $\fQ \to \bbP(W)$ is 1-regular in the sense of Debarre--Kuznetsov at all points of $S$, and by \cite[Lemma 3.9]{ogsmoothness} the total space of the orthogonal Grassmannian fibration $\OGr(k+2,\fQ)$ is smooth of the expected dimension $(k+2)(n-k+1) - {k+3 \choose 2} + 1$. \par 
	As before, we show that this fits into the regime of Situation \ref{flippingsituation}. We take: 
	\begin{itemize}
		\item $G = \Gr(k+2,V)$, with vector bundles $\Sym^2 \cU_{k+2}^\vee$ and $W \otimes \cO$ (in the sequel, we will write this simply as $W$ for convenience),
		\item $Y = G_k(X)$ and $Y' = \OGr(k+2,\fQ)$, with inclusions $G_k(X) \hookrightarrow G_k(\bbP(V)) \simeq \bbP_G(\Sym^2 \cU_{k+2}^\vee)$ and $\OGr(k+2,\fQ) \hookrightarrow \Gr(k+2,V) \times \bbP(W)$, and 
		\item $F = F_{k+1}(X)$, with $Z := Y \times_G F \simeq \bbP(\Sym^2 \cU_{k+2}^\vee|_F)$ the quadrics whose linear spans lie within $X$, and $Z' := Y' \times_G F \simeq F_{k+1}(X) \times \bbP(W)$ consisting of isotropic planes for the pencil which are simultaneously isotropic for every quadric in $W$. 
	\end{itemize}
	\begin{lemma}\label{smoothnesslemmaquartic}
		The variety $G_k(X)$ is smooth along the subvariety $Z$. 
	\end{lemma}
	\begin{proof}
		We proceed in a similar manner to the case of cubics. Let $\Sigma \in Z$, and observe that the $(k+1)$-plane $\langle \Sigma \rangle$ spanned by $\Sigma$ must lie inside of $X$. As before, since $\Sigma \subset X$ is a local complete intersection, to check the smoothness of $G_k(X)$ along $Z$, it suffices to check that $h^1(\cN_{\Sigma/X}) = 0$. \par 
		We observe again the existence of two normal bundle sequences 
		\begin{equation}
			\begin{gathered}
				0 \to \cN_{\Sigma/\langle \Sigma \rangle} \simeq \cO_{\Sigma}(2) \to \cN_{\Sigma/X} \to \cN_{\langle \Sigma \rangle/X}|_\Sigma \to 0 \\
				0 \to \cN_{\langle \Sigma \rangle/X} \to \cN_{\langle \Sigma \rangle/\bbP^{n+2}} \simeq \cO_{\langle \Sigma \rangle}(1)^{\oplus(n-k+1)} \to \cN_{X/\bbP^{n+2}}|_{\langle \Sigma \rangle} \simeq \cO_{\langle \Sigma \rangle}(2)^{\oplus2} \to 0.
				\label{nmbdlseqquartic1}
			\end{gathered}
		\end{equation}
		Using that the Fano variety of $(k+1)$-planes on $X$ is smooth of the expected dimension, it follows that 
		\[\dim T_{F_{k+1}(X),\langle \Sigma \rangle} = h^0(\cN_{\langle \Sigma \rangle/X}) = (k+2)(n-k+1) - 2{k+3 \choose 2},\]
		and hence by counting dimensions we can verify the identity 
		\[h^0(\cN_{\langle \Sigma \rangle/X}) + h^0(\cN_{X/\bbP^{n+2}}|_{\langle \Sigma \rangle}) = h^0(\cO_{\langle \Sigma \rangle}(1))^{\oplus(n-k+1)}.\]
		It follows that the last map in the sequence on $H^0$ coming from the second sequence of \eqref{nmbdlseqquartic1} 
		\[0 \to H^0(\cN_{\langle \Sigma \rangle/X}) \to H^0(\cO_{\langle \Sigma \rangle}(1))^{\oplus(n-k+1)} \to H^0(\cO_{\langle \Sigma \rangle}(2))^{\oplus 2}\]
		is surjective, and so $h^1(\cN_{\langle \Sigma \rangle/X}) = 0$. Twisting the second sequence of \eqref{nmbdlseqquartic1} by $(-2)$ and taking cohomology, we see that 
		\[0 = H^1(\cO_{\langle \Sigma \rangle})^{\oplus 2} \to H^2(\cN_{\langle \Sigma \rangle/X}(-2)) \to H^0(\cO_{\langle \Sigma \rangle}(-1))^{\oplus(n-k+1)} = 0,\]
		forcing $h^2(\cN_{\langle \Sigma \rangle}(-2)) = 0$. Then using the restriction sequence 
		\[0 \to \cN_{\langle \Sigma \rangle/X}(-2) \to \cN_{\langle \Sigma \rangle/X} \to \cN_{\langle \Sigma \rangle/X}|_\Sigma \to 0,\]
		we find that $h^1(\cN_{\langle \Sigma \rangle/X}|_\Sigma) = 0$, and as $h^1(\cO_\Sigma(2)) = 0$ we conclude using the first sequence of \eqref{nmbdlseqquartic1} that $h^1(\cN_{\Sigma/X}) = 0$ as well. 
	\end{proof}
	Next, we construct a map $\textbf{s} : \cO_{Y'}(1) \to \Sym^2 \cU_{k+2}^\vee$. By construction on $Y' = \OGr(k+2,\fQ)$ we have a tautological subbundle $\cO_{Y'}(-1) \to W$ pulled back from $\bbP(W)$ with cokernel $\cO_{Y'}(1)$. If we compose with the restriction map $W \to \Sym^2 V^\vee \to \Sym^2 \cU_{k+2}^\vee$, the isotropy condition defining $\OGr(k+2,\fQ)$ implies that the composition 
	\[\cO_{Y'}(-1) \to W \to \Sym^2 \cU_{k+2}^\vee\]
	must vanish, and so the map $W \to \Sym^2 \cU_{k+2}^\vee$ factors through the cokernel of the first map. In particular, since the cokernel is isomorphic to $\cO_{Y'}(1)$, we get a factorization 
	\begin{equation}
		W \to \cO_{Y'}(1) \to \Sym^2 \cU_{k+2}^\vee \label{quarticsectiondef1}
	\end{equation}
	and then define $\textbf{s} : \cO_{Y'}(1) \to \Sym^2 \cU_{k+2}^\vee$ as our desired map.\par 
	Geometrically, for a point $(\bbP(U_{k+2}) \subset Q) \in \OGr(k+2,\fQ)$, $\textbf{s}$ picks out the quadric on $\bbP(U_{k+2})$ coming from taking any quadric equation in the pencil $\langle Q_1, Q_2 \rangle$ distinct from $Q$ and restricting it to $\bbP(U_{k+2})$. 
	\begin{lemma} \label{sectionlemmaquartics}
		The section $\textbf{s}$ is injective with locally free cokernel along $U'$, and induces an isomorphism onto $U$. Moreover, $Z' = V(\textbf{s})$ has codimension equal to $\on{rank} \Sym^2 \cU_{k+2}^\vee$. 
	\end{lemma}
	\begin{proof}
		The locus $U' \subset Y'$ is the complement of $Z' = Y' \times_G F$, so in particular the defining equations \[W \to \Sym^2 \cU_{k+2}^\vee\]
		for $F_{k+1}(X) \subset G = \Gr(k+2,V)$ is fiberwise nonvanishing. Hence the map $\textbf{s} : \cO_{Y'}(1) \to \Sym^2 \cU_{k+2}^\vee$, which factors this map, is also fiberwise nonvanishing, and so must be injective with locally free cokernel along $U'$. Therefore we may define a map $U' \to \bbP(\Sym^2 \cU_{k+2}^\vee)$ induced by $\textbf{s}$; but by construction, since the associated family of quadrics lies within $X$, the map must land in $U \subset G_k(X)$. \par 
		We next construct the inverse map $U \to \Gr(k+2,V) \times \bbP(W)$. On $Y = G_k(X)$, we may again consider the restriction map 
		\[W \to \Sym^2 \cU_{k+2}^\vee.\]
		Since the universal quadric over $G_k(X)$ lies inside of every quadric in the pencil, it follows that the map factors as 
		\[W \to \cO_{Y}(-1) \hookrightarrow \Sym^2 \cU_{k+2}^\vee,\]
		where the second map is the tautological subbundle coming from the inclusion $G_k(X) \hookrightarrow \bbP(\Sym^2 \cU_{k+2}^\vee)$. Over $U$, the map $W \to \cO_Y(-1)$ restricts to a surjective map, as fiberwise for points in $U$, the composition $W \to \cO_Y(-1) \hookrightarrow \Sym^2 \cU_{k+2}^\vee$ does not vanish. In particular one can observe that $\ker(W \to \cO_Y(-1)|_U)$ is necessarily a line bundle and must be identified with $\cO_Y(1)|_U$, so we get a map
		\begin{equation}
			\cO_{Y}(1)|_U = \ker(W \to \cO_Y(-1)|_U) \hookrightarrow W \label{quarticsectiondef2}
		\end{equation}
		which is an injection with cokernel $\cO_Y(1)|_U$, and take this to be our definition of $\textbf{s}' : \cO_{Y}(1)|_U \to W$. This then defines a map $U \to \Gr(k+2,V) \times \bbP(W)$. The condition that the image lies in $\OGr(k+2,\fQ)$ follows since the composition $\cO_Y(1)|_U \xrightarrow{\textbf{s}'} W \to \Sym^2 \cU_{k+2}^\vee|_U$ vanishes by construction.  \par 
		To check that the compositions are the identity is a straightforward exercise, analogous to the case of cubics. For example, to check that the composition $U \to U' \to U$ is the identity, it suffices to check that the pullback of $\textbf{s} : \cO_{Y'}(1) \to \Sym^2 \cU_{k+2}^\vee$ along the map $U \to U'$ induced by $\textbf{s}'$ is the restriction to $U$ of the tautological section $\cO_Y(-1) \hookrightarrow \Sym^2 \cU_{k+2}^\vee$. But over $U$, both maps are identified with the inclusion of the image $\im(W \to \Sym^2 \cU_{k+2}^\vee) \hookrightarrow \Sym^2 \cU_{k+2}^\vee$, so they must coincide. The other composition is similar. \par 
		To verify the equality $V(\textbf{s}) = Z'$: $Z'$ is cut out by the pullback to $Y'$ of the defining map $W \to \Sym^2 \cU_{k+2}^\vee$ of $F$, and by the factorization of \eqref{quarticsectiondef1}, that map factors as 
		\[ W \to \cO_{Y'}(1) \xrightarrow{\textbf{s}} \Sym^2 \cU_{k+2}^\vee\]
		where the first map is surjective. Evidently as a surjective map cannot vanish at any point of $Y'$, it follows that the vanishing of the composition is equivalent to the vanishing of $\textbf{s}$, whence the claim. \par 
		The final claim on the codimension of $Z' =F_{k+1}(X)\times \bbP(W)$ comes from the following identities: 
		\begin{equation}
			\begin{aligned}
				\dim Z' &= (k+2)(n-k+1) - 2{k+3 \choose 2} + 1 \\ 
				&= (k+2)(n-k+1) - {k+3 \choose 2} + 1 - {k + 3 \choose 2} \\
				&= \dim Y' - \on{rank} \Sym^2 \cU_{k+2}^\vee. 
			\end{aligned}
		\end{equation}
	\end{proof}
	\begin{proof}[Proof of Theorem \ref{geomflipquartic}]
		Combine Lemmas \ref{flippinglemma}, \ref{smoothnesslemmaquartic} and \ref{sectionlemmaquartics}. 
	\end{proof}
	\subsection{The case of linear sections of $\Gr(2,V_5)$}
	As in Theorem \ref{geomflipthm}(c), we let $X = \Gr(2,V_5) \cap \bbP(W) \subset \bbP(\bigwedge^2 V_5)$ be a smooth transverse linear section of $\Gr(2,V_5)$ embedded via the Pl\"ucker embedding. Let $n := \dim X$. \par 
	\begin{theorem}
		\label{geomflipquintic}
		 Let $V_5$ be a 5-dimensional vector space, and suppose $X = \Gr(2,V_5) \cap \bbP(W) \subset \bbP(\bigwedge^2 V_5)$ is a smooth transverse linear section for some $\bbP(W) \subset \bbP(\bigwedge^2 V_5)$, and suppose $\dim X = n$. Let $0 \leq k \leq n.$ Let $\cU_{k+2}$ denote the tautological subbundle of rank $k+2$ on $\Gr(k+2,V)$, and $\cR_4$ denote on $\Gr(4,V_5)$. Then $\cW \coloneqq\bigwedge^2 \cR_4 \cap W$ is a vector bundle of rank $n$ on $\Gr(4,V_5) \simeq \bbP(V_5^\vee)$. If $k = 0$ or $1$, we have a standard flip diagram 
		\begin{equation}
			\begin{tikzcd}[sep=small]
				\bbP(\Sym^2 \cU_{k+2}^\vee|_{F_{k+1}(X)}) \arrow[d] \arrow[r, hook]& G_k(X) \arrow[leftrightarrow, r, dashed, "\text{flip}"] & \Gr_{\Gr(4,V_5)}(k+2,\bigwedge^2\cR_4\cap W) \arrow[d] & \arrow[l, hook'] F_{k+1}(X) \times \Gr(4,V_5) \arrow[d]\\
				F_{k+1}(X) & &\Gr(4,V_5) & F_{k+1}(X)
			\end{tikzcd}
		\end{equation}
		In particular, $G_k(X)$ is smooth and irreducible of the expected dimension. \par 
		If $k = 2$, then there is instead a decomposition 
		\begin{equation}
			G_2(X) \simeq \Gr_{\Gr(4,V_5)}(4,\bigwedge^2 \cR_4 \cap W) \sqcup \bbP(\Sym^2 \cU_{4}^\vee),
		\end{equation}
		so that $G_2(X)$ is the disjoint union of two smooth irreducible subvarieties. For $k > 2$, there is an isomorphism 
		\begin{equation}
			G_k(X) \simeq \Gr_{\Gr(4,V_5)}(k+2,\bigwedge^2 \cR_4 \cap W).
		\end{equation}
	\end{theorem}
	Before proving the theorem, we establish some background on the geometry of $\Gr(2,V_5)$ and its related varieties, which is somewhat more involved than the previous two cases. 
	\subsubsection{Linear spaces on linear sections of $\Gr(2,V_5)$}
	Let us first recall some features of the geometry of the Fano schemes of linear spaces on $X$ \cite[Section 3.1]{GMquadrics}. In particular, we recall the two types of linear spaces on $\Gr(2,V_5)$: 
	\begin{defn}
		Suppose $P \subset \Gr(2,V_5) \subset \bbP(\bigwedge^2 V_5)$ is a linear subspace. \begin{itemize}
			\item $P$ is a \textbf{$\sigma$-space} if $P \subset \bbP(V_1 \wedge V_5) \subset \Gr(2,V_5)$ for some 1-dimensional subspace $V_1 \subset V_5$. 
			\item $P$ is a \textbf{$\tau$-space} if $P \subset \bbP(\bigwedge^2 V_3) \subset \Gr(2,V_5)$ for some 3-dimensional subspace $V_3 \subset V_5$. 
		\end{itemize}
		
	\end{defn}
	Let $F_k^\sigma(X)$ (resp. $F_k^\tau(X)$) denote the closed subscheme of the Fano variety $F_k(X)$ parametrizing $\sigma$-planes (resp. $\tau$-planes). 
	\begin{proposition}\label{fanodims}
		For $k \geq 3$, $F_{k+1}(X)$ is empty. For $0 \leq k \leq 2$, the $F_{k+1}(X)$ are smooth whenever nonempty, and have dimension given by the following table: 
		\begin{center}
		\begin{tabular}{c|c|c|c|c|c}
			$\dim X$ & $2$ & $3$ & $4$ & $5$ & $6$ \\\hline
			$\dim F_1(X)$& $0$ & $2$ & $4$ & $6$ & $8$ \\
			$\dim F_2^\sigma(X)$& $\varnothing$ &$\varnothing$ & $1$ & $4$ & $7$ \\
			$\dim F_2^\tau(X)$& $\varnothing$  & $\varnothing$& $0$ & $3$ & $6$ \\
			$\dim F_3(X)$& $\varnothing$  &$\varnothing $&$\varnothing$ & $0$ & $4$
		\end{tabular}
		\end{center}
		Moreover, $F_1(X) = F_1^\sigma(X) = F_1^\tau(X)$, $F_2(X) = F_2^\sigma(X) \sqcup F_2^\tau(X)$, and $F_3(X) = F_3^\sigma(X)$ with $F_3^\tau(X) = \varnothing$. 
	\end{proposition}
	\begin{proof}
		Besides the dimensions in the table for $\dim X = 2$, everything follows from Lemma 3.1 of \cite{GMquadrics} and the subsequent statements in that paper. For $\dim X = 2$, $X$ is a quintic del Pezzo surface, and it is classical that $X$ contains finitely many lines. 
	\end{proof}
	We will also find it useful to identify the restrictions of the normal bundle $\cN_{X/\bbP(W)}$ to $\sigma$-planes and $\tau$-planes. 
	\begin{lemma} \label{normalrestriction}
		Let $\bbP(L) = P \subset X$ be a linear subspace of dimension $k + 1$. The normal bundle $\cN_{X/\bbP(W)}$ is given by $\cQ^\vee(2)$ where $\cQ$ is the restriction of the tautological quotient bundle of $\Gr(2,V_5)$ to $X$. \begin{itemize}
			\item If $P$ is a $\sigma$-plane, then $\cN_{X/\bbP(W)}|_P \simeq (\cQ_P^\vee \oplus \cO_P^{\oplus (2-k)})\otimes\cO_P(2)$ where $\cQ_P$ is the tautological quotient bundle $L/\cO_P(-1)$. 
			\item If $P$ is a $\tau$-plane, then $\cN_{X/\bbP(W)}|_P \simeq \cO_P(1)\oplus \cO_P(2)^{\oplus 2}.$ 
		\end{itemize}
	\end{lemma}
	\begin{proof}
		The isomorphism $\cN_{\Gr(2,V_5)/\bbP(\bigwedge^2 V_5)} \simeq \cQ^\vee(2)$ is proven in \cite[Lemma A.2]{bcp}. \par 
		Suppose first that $P = \bbP(L)$ is a $\sigma$-plane, so that there exists $V_1$ such that $\bbP(L) \subset \bbP(V_1 \wedge V_5) \subset \Gr(2,V_5)$. Choosing a complementary subspace $V_4$ to $V_1 \subset V_5$, we may fix a splitting $V_5 = V_1 \oplus V_4$. Fixing any nonzero $v_1 \in V_1$, we then have a natural identification 
		\begin{equation}
			\begin{aligned}
				V_4 &\rightarrow V_1 \wedge V_5 \\
				v &\mapsto v_1 \wedge v
			\end{aligned}
		\end{equation}
		The associated map $\bbP(V_4) \to \Gr(2,V_5)$ sends $\on{span}(v) \mapsto \on{span}(v_1, v)$. There is then a natural map of short exact sequences:  
		\begin{equation}
			\begin{tikzcd}
				0 \arrow[r] & \cO_{\bbP(V_4)}(-1) \arrow[r] \arrow[d, hook]  & V_4  \arrow[r] \arrow[d, hook] & \cQ_{\bbP(V_4)} \arrow[r] \arrow[d] & 0 \\
				0 \arrow[r] & \cU|_{\bbP(V_4)} \arrow[r] & V_5\arrow[r] &\cQ|_{\bbP(V_4)} \arrow[r] & 0
			\end{tikzcd}
		\end{equation}
		where $\cU$ and $\cQ$ denote the tautological subbundle and quotient bundle of $\Gr(2,V_5)$. It is an easy consequence of the snake lemma that $\cQ|_{\bbP(V_4)} \simeq \cQ_{\bbP(V_4)}$. On the other hand, a similar argument via the snake lemma shows that $\cQ_{\bbP(V_4)}|_{\bbP(L)} \simeq \cQ_{\bbP(L)} \oplus \cO^{\oplus 3 - (k+1)}$. The claim on the normal bundle then follows. \par 
		Suppose now that $P = \bbP(L)$ is a $\tau$-plane, so that there exists $V_3$ such that $\bbP(L) \subset \bbP(\bigwedge^2 V_3) \subset \Gr(2,V_5)$. Note that under the Pl\"ucker map, $\Gr(2,V_3) \simeq \bbP(\bigwedge^2 V_3)$. Then in particular, we have a map of short exact sequences given by 
		\begin{equation}
			\begin{tikzcd}
				0 \arrow[r] & \cU_{\Gr(2,V_3)}|_{\bbP(L)} \arrow[r] \arrow[d, "\sim" {anchor = south, rotate = 90}]  & V_3  \arrow[r] \arrow[d, hook] & \cO_{\bbP(L)}(1) \arrow[r] \arrow[d] & 0 \\
				0 \arrow[r] & \cU|_{\bbP(L)} \arrow[r] & V_5\arrow[r] &\cQ|_{\bbP(L)} \arrow[r] & 0
			\end{tikzcd}
		\end{equation}
		which yields by the snake lemma a split short exact sequence in the right column giving an isomorphism $\cQ|_{\bbP(L)} \simeq \cO_{\bbP(L)}(1) \oplus \cO^{\oplus2}$. The result on the normal bundle follows from this isomorphism. 
	\end{proof}
	\subsubsection{Geometry of the flip}
	First recall that one may identify the space of quadrics containing $\Gr(2,V_5)$ with $V_5$ under the map 
	\begin{align*}
		V_5 &\to \Hom(\Sym^2 \bigwedge^2 V_5, \det V_5) \simeq \Sym^2 \bigwedge^2 V_5^\vee \\
		v &\mapsto (\alpha \mapsto \alpha \wedge \alpha \wedge v)
	\end{align*}
	Recall also that we define the coherent sheaf $\cW := \bigwedge^2 \cR_4 \cap W$ on $\Gr(4,V_5)$; by \cite[Lemma 4.1]{GMquadrics} this is in fact a vector bundle of rank $n$. Then there is a diagram 
	\begin{equation}
		\begin{tikzcd}
			G_k(X) \arrow[rd] & & \arrow[ld] \Gr_{\Gr(4,V_5)}(k+2,\cW) \\
			& \Gr(k+2,W) &
		\end{tikzcd}
	\end{equation}
	where the left map, as usual, sends a quadric $\Sigma$ to the $(k+1)$-plane $\langle \Sigma \rangle$ it spans, and where the right map is given by the composition of the inclusion \[\Gr_{\Gr(4,V_5)}(k+2,\cW) \subset \Gr(4,V_5) \times \Gr(k+2,W)\] followed by projection onto the second factor. \par
	We can identify the defining equations for $\Gr_{\Gr(4,V_5)}(k+2,\cW)  \subset \Gr(4,V_5) \times \Gr(k+2,W)$ as follows. First, observe that the natural coevaluation map 
	\[V_5^\vee \to V_5 \otimes V_5^\vee \otimes V_5^\vee \to V_5 \otimes \bigwedge^2 V_5^\vee\]
	can be interpreted as the map sending a hyperplane $R_4 \subset V_5$ to the collection of hyperplanes on $\bigwedge^2 V_5$ containing $\bigwedge^2 R_4$; then the condition on $\Gr(4,V_5) \times \Gr(k+2,W)$ of lying in the subvariety $\Gr_{\Gr(4,V_5)}(k+2,\cW)$ is equivalent to asking that the composite map 
	\begin{equation}
		\cO_{\Gr(4,V_5)}(-1) \hookrightarrow V_5^\vee \hookrightarrow V_5 \otimes \bigwedge^2 V_5^\vee \to V_5 \otimes W^\vee \to V_5 \otimes \cU_{k+2}^\vee \label{quinticeqndef}
	\end{equation} 
	vanishes. \par 
	Now there is a natural homomorphism of bundles on $\Gr(k+2,W)$ given by the composition $V_5 \to \Sym^2 \bigwedge^2 V_5^\vee \to \Sym^2 \cU_{k+2}^\vee$ consisting of the restrictions of the Pl\"ucker quadrics to the tautological subbundle of $\Gr(k+2,W)$; the vanishing of this map gives $F_{k+1}(X)$. 
	\begin{lemma} \label{quinticprojbundlelemma}
		For $k = 0$, $\Gr_{\Gr(4,V_5)}(k+2,\cW) \times_{\Gr(k+2,W)} F_{k+1}(X)$ is a projective bundle over the components of $F_1(X)$. For $k = 1$, it is an equidimensional union of projective bundles over the components of $F_2(X)$. For $k = 2$, it is empty. 
	\end{lemma}
	\begin{proof}
		For $k = 0, 1$, we check that this map is a projective bundle by showing that over $F_{k+1}(X)$ the map $\phi : V_5^\vee \to V_5 \otimes \cU_{k+2}^\vee$ has constant rank; since $F_{k+1}(X)$ is smooth and therefore reduced, this will imply that the kernel $\cK \subset V_5$ is a locally free subsheaf with locally free cokernel, and by restricting the defining equation \eqref{quinticeqndef} for $\Gr_{\Gr(4,V_5)}(k+2,\cW)$ in $\Gr(4,V_5) \times \Gr(k+2,W)$ to $F_{k+1}(X)$ it follows that the condition of a point lying on the fiber product is equivalent to asking that the map $\cO_{\Gr(4,V_5)}$ factors through $\cK$. In particular, this implies an identification \[\Gr_{\Gr(4,V_5)}(k+2,\cW) \times_{\Gr(k+2,W)} F_{k+1}(X) \simeq \bbP(\cK).\] \par 
		Fix now a subspace $U_{k+2} \subset W \subset \bigwedge^2 V_5$ such that $\bbP(U_{k+2}) \in F_{k+1}^\sigma(X)$; since we have assumed $U_{k+2}$ is a $\sigma$-space, it follows that there exist linear subspaces $V_1 \subset V_{k+3} \subset V_5$ of dimensions $1$ and $k + 3$ respectively such that $U_{k+2} = V_1 \wedge V_{k+3}$. For a linear functional $\psi \in V_5^\vee$ to vanish under $\phi$ at this point is equivalent to the inclusion $V_{k+3} \subset \ker \psi$. In particular, we see that $\ker \phi = V_{k+3}^\perp \subset V_5^\vee$, implying that $\dim \ker \phi = 5 - (k+3)$; hence $\phi$ is of constant rank over $F_{k+1}^\sigma$, and in fact the fiber of $\cK$ is precisely $\ker \phi$. \par 
		Suppose we take $U_{k+2} \subset W \subset \bigwedge^2 V_5$ to lie in $F_{k+1}^\tau(X) \setminus F_{k+1}^\sigma(X)$; of course, by Proposition \ref{fanodims} this can only happen when $k = 1$. Then by this assumption, there exists a 3-dimensional $V_3 \subset V_5$ such that $\bbP(U_{3}) \subset \bbP(\bigwedge^2 V_3) \subset \Gr(2,V_5)$; in this case we must have $U_3 = \bigwedge^2 V_3$. Then a linear functional $\psi \in V_5^\vee$ will vanish under $\phi$ if and only if $V_3 \subset \ker \psi$, or equivalently $\ker \phi = V_3^\perp$, which will be of dimension $5 - 3 = 2$. Hence again $\phi$ is of constant rank, and the fiber of $\cK$ is this subspace. \par 
		For $k = 0$, as $F_1(X) = F_1^\sigma(X)$ is always connected, our desired result follows and we see that $\cK$ is of rank 2. For $k = 1$, over $F_{2}^\sigma(X)$ we see that $\cK$ is of rank 1, and over $F_{2}^\tau$, $\cK$ is of rank 2. Then \[\dim \bbP(\cK) \times_{\Gr(3,W)} F_2^\sigma(X) = \dim F_2^\sigma(X) = \dim F_2^\tau(X) + 1 = \dim \bbP(\cK) \times_{\Gr(3,W)} F_2^\tau(X), \]
		and hence $\Gr_{\Gr(4,V_5)}(3,\cW) \times_{\Gr(3,W)} F_{2}(X)$ is indeed equidimensional. \par 
		For $k = 2$, the emptiness is a consequence of our calculation for the fibers of the kernel of $\phi$ for $\sigma$-spaces; indeed, the calculation above shows that $\phi$ is necessarily an injective map and hence the defining equation \eqref{quinticeqndef} does not vanish over $F_{3}(X)$. 
	\end{proof}
	We are now in a position to show that we are in of Situation \ref{flippingsituation}. We take: 
	\begin{itemize}
		\item $G = \Gr(k+2,W)$, with vector bundles $\Sym^2 \cU_{k+2}^\vee$ and $V_5^\vee \otimes \cO$ (as before, we will write this simply as $V_5^\vee)$, 
		\item $Y = G_k(X)$ and $Y' = \Gr_{\Gr(4,V_5)}(k+2,\cW)$, with inclusions $G_k(X) \hookrightarrow G_k(\bbP(W)) \simeq \bbP_G(\Sym^2 \cU_{k+2}^\vee)$ and $\Gr_{\Gr(4,V_5)}(k+2,\cW) \hookrightarrow \Gr(k+2,W) \times \bbP(V_5^\vee)$, and 
		\item $F = F_{k+1}(X)$, with $Z := Y \times_G F \simeq \bbP(\Sym^2 \cU_{k+2}^\vee|_F)$ the quadrics whose linear spans lie within $X$, and $X' := Y' \times_G F$ which is either a projective bundle, an equidimensional union of projective bundles, or empty by Lemma \ref{quinticprojbundlelemma}.
	\end{itemize}
	\begin{lemma} \label{smoothnesslemmaquintics}
		The variety $G_k(X)$ is smooth along the subvariety $Z$. 
	\end{lemma}
	\begin{proof}
		The strategy is essentially the same as the cases of cubics or intersections of two quadrics. As before, it suffices to compute $H^1(\cN_{\Sigma/X}) = 0$ for $\Sigma \in G_k(X)$ such that $\langle \Sigma \rangle \subset X$. As per usual, we make use of normal bundle sequences 
		\begin{equation}
			\label{nmbdlseqquintics1}
			\begin{gathered}
				0 \to \cN_{\Sigma/\langle \Sigma \rangle} \simeq \cO_{\Sigma}(2) \to \cN_{\Sigma/X} \to \cN_{\langle \Sigma \rangle/X}|_{\Sigma} \to 0 \\
				0 \to \cN_{\langle \Sigma \rangle / X} \to \cO_{\langle \Sigma \rangle}(1)^{\oplus(n-k+2)} \to \cN_{X/\bbP(W)}|_{\langle \Sigma \rangle} \to 0 
			\end{gathered}
		\end{equation}
		and observe as before that to check $H^1(\cN_{\langle \Sigma \rangle/X}) = 0$, it is enough to check that the \[H^0(\cO_{\langle \Sigma \rangle}(1))^{\oplus(n-k+2)} \to H^0(\cN_{X/\bbP(W)}|_{\langle \Sigma \rangle})\] coming from the second sequence of \eqref{nmbdlseqquintics1} is surjective. In fact, once we know this we are essentially done; indeed, twisting the second sequence of \eqref{nmbdlseqquintics1} by $(-2)$ and taking cohomology, we find an exact sequence
		\[H^1(\cN_{X/\bbP(W)}(-2)|_{\langle \Sigma \rangle}) \to H^2(\cN_{\langle \Sigma \rangle/X}(-2)) \to H^2(\cO_{\langle\Sigma\rangle}(-1))^{\oplus(n-k+2)} = 0.\]
		But by Lemma \ref{normalrestriction}, either \[\cN_{X/\bbP(W)}(-2)|_{\langle \Sigma \rangle} \simeq \cQ_{\langle \Sigma \rangle}^\vee \oplus \cO_{\langle \Sigma \rangle}^{\oplus(2-k)}\] or \[\cN_{X/\bbP(W)}(-2)|_{\langle \Sigma \rangle} \simeq \cO_{\langle \Sigma \rangle}(-1) \oplus \cO_{\langle \Sigma \rangle}^{\oplus 2},\]
		and in either case it is easy to see that $H^1(\cN_{X/\bbP(W)}(2)|_{\langle \Sigma \rangle}) = 0$. Then using the restriction sequence 
		\begin{equation}
			0 \to \cN_{\langle \Sigma \rangle/X}(-2) \to \cN_{\langle \Sigma \rangle/X} \to \cN_{\langle \Sigma \rangle/X}|_{\Sigma} \to 0
		\end{equation}
		we see that $H^1(\cN_{\langle \Sigma \rangle/X}|_{\Sigma}) = 0$, and combining with the first sequence of \eqref{nmbdlseqquintics1} we conclude that $H^1(\cN_{\Sigma/X}) = 0$. \par 
		We return now to showing the surjectivity of $H^0(\cO_{\langle \Sigma \rangle}(1))^{\oplus(n-k+2)} \to H^0(\cN_{X/\bbP(W)}|_{\langle \Sigma \rangle})$, which we can check by verifying 
		\begin{equation}
			h^0(\cN_{\langle \Sigma \rangle/X})  + h^0(\cN_{X/\bbP(W)}|_{\langle \Sigma \rangle}) = h^0(\cO_{\langle \Sigma \rangle}(1))^{\oplus(n-k+2)}. \label{surjectivityidentity}
		\end{equation}
		Assume first that $\langle \Sigma \rangle$ is a $\sigma$-plane. Then by Proposition \ref{fanodims}, $h^0(\cN_{\langle \Sigma \rangle/X}) = \dim F_{k+1}^\sigma(X)$. On the other hand, by Lemma \ref{normalrestriction}, $\cN_{X/\bbP(W)}|_{\langle \Sigma \rangle} = \cQ_{\langle \Sigma \rangle}^\vee(2) \oplus \cO_{\langle \Sigma \rangle}(2)^{\oplus(2-k)}$. But a calculation from the cohomology of line bundles on projective space gives $h^0(\cO_{\langle \Sigma \rangle}(2)) = {k+3 \choose 2}$, and it follows from the Borel--Weil--Bott theorem or the long exact sequence on cohomology for a twist of the Euler sequence that 
		\[h^0(\cQ^\vee_{\langle \Sigma \rangle}(2)) = \frac{(k+1)(k+2)(k+3)}{3}\]
		and so 
		\[h^0(\cN_{X/\bbP(W)}|_{\langle \Sigma \rangle}) = \frac{(k+1)(k+2)(k+3)}{3} + (2-k){k+3 \choose 2}. \]
		It is then easy to check case by case using Proposition \ref{fanodims} that for $k = 0, 1$ or $2$ (when $F_{k+1}^\sigma(X)$ is nonempty), the identity
		\[\dim F_{k+1}^\sigma(X) = (n-k+2)(k+2) - \frac{(k+1)(k+2)(k+3)}{3} - (2-k){k+3 \choose 2}\]
		holds, which implies that the identity for $h^0$ in \eqref{surjectivityidentity} holds. \par 
		Assume now that $\langle \Sigma \rangle$ is a $\tau$-plane. As before, to show \eqref{surjectivityidentity} it reduces to showing the identity 
		\[\dim F_{k+1}^\tau(X) = (n-k+2)(k+2) - h^0(\cN_{X/\bbP(W)}|_{\langle \Sigma \rangle})\]
		whenever $F_{k+1}^\tau(X)$ is nonempty, and in this case Lemma \ref{normalrestriction} shows that $\cN_{X/\bbP(W)}|_{\langle \Sigma \rangle} = \cO_{\langle \Sigma\rangle}(1) \oplus \cO_{\langle \Sigma \rangle}(2)^{\oplus 2}$, so 
		\[h^0(\cN_{X/\bbP(W)}|_{\langle \Sigma \rangle}) = (k+2) + 2{k+3 \choose 2}.\]
		As above, when $k = 0$ or $1$ (when $F_{k+1}^\tau(X)$ is nonempty), the identity 
		\[\dim F_{k+1}^\sigma(X) = (n-k+2)(k+2) - (k+2) - 2{k+3 \choose 2}\]
		holds by a case by case analysis of Proposition \ref{fanodims}. \par 
	\end{proof}
	To construct our map $\textbf{s} : \cO_{Y'}(1) \to \Sym^2 \cU_{k+2}^\vee$ on $Y'  = \Gr_{\Gr(4,V_5)}(k+2,\cW)$, we begin with the map 
	\[\cR_4 \hookrightarrow V_5 \hookrightarrow \Sym^2 V_5^\vee \twoheadrightarrow \Sym^2 \cW^\vee \twoheadrightarrow \Sym^2 \cU_{k+2}^\vee,\]
	coming from composition of the tautological inclusion $\cR_4 \hookrightarrow V_5$ coming from $\Gr(4,V_5)$ with the restriction of Pl\"ucker quadrics. By \cite[Lemma 4.3]{GMquadrics}, the composition $\cR_4 \to \Sym^2 \cW^\vee$ vanishes, so in particular the map $V_5 \to \Sym^2 \cW^\vee \to \Sym^2 \cU_{k+2}^\vee$ factors as 
	\begin{equation}
		V_5 \to \cO_{Y'}(1) \to \Sym^2 \cW^\vee \twoheadrightarrow \Sym^2 \cU_{k+2}^\vee,
	\end{equation}
	and we define $\textbf{s} : \cO_{Y'}(1) \to \Sym^2 \cU_{k+2}^\vee$ to be the composition of the second two maps in the factorization above. \par 
	Geometrically, for a point $(R_4, U_{k+2}) \in \Gr_{\Gr(4,V_5)}(k+2,\cW)$, $\textbf{s}$ picks out the quadric on $\bbP(U_{k+2})$ coming from taking a Pl\"ucker quadric not in $R_4$ and restricting it to $\bbP(U_{k+2})$. 
	\begin{lemma} \label{sectionlemmaquintics}
		The map $\textbf{s}$ is injective with locally free cokernel along $U'$, and induces an isomorphism onto $U$. Moreover, when $Z' = V(\textbf{s})$ is nonempty, it has codimension equal to $\on{rank}\Sym^2 \cU_{k+2}^\vee$. 
	\end{lemma}
	\begin{proof}
		Recall that $U'$ is defined as the complement of $F$ on $\Gr(k+2,W)$, where $F$ is cut out by the vanishing of the map $V_5 \to \Sym^2 \cU_{k+2}^\vee$. \par 
		As in the case of the intersection of two quadrics, to check $\textbf{s}$ is injective with locally free cokernel along $U'$ it suffices to observe that the map $V_5 \to \Sym^2 \cU_{k+2}^\vee$ is fiberwise nonzero along $U'$; this implies that the map $\cO_{Y'}(1) \to \Sym^2 \cU_{k+2}^\vee$ is fiberwise injective and therefore injective with locally free cokernel. In particular, we have the existence of a map $U' \to \bbP(\Sym^2 \cU_{k+2}^\vee)$, which as before lands in $G_k(X)$ (and moreover in $U$) as the associated family of quadrics by construction lies inside of $X$. \par 
		To produce the inverse map, consider again the map 
		\[V_5 \to \Sym^2 \cU_{k+2}^\vee.\]
		Since each Pl\"ucker quadric contains the universal quadric over $G_k(X)$, this map thus factors as 
		\[V_5 \to \cO_Y(-1) \hookrightarrow \Sym^2 \cU_{k+2}^\vee,\]
		and we take $\textbf{s}' : \cO_Y(1) \to V_5^\vee$ to be the dual of the first map; as $V_5 \to \Sym^2 \cU_{k+2}^\vee$ is nonzero over $U$, the map $V_5 \to \cO_Y(-1)$ is necessarily a surjection and hence $\textbf{s}'$ is an injection with locally free cokernel over $U$. As a consequence, $\textbf{s}'$ defines a map \[U \to \bbP(V_5^\vee) \times \Gr(k+2,W) = \Gr(4,V_5) \times \Gr(k+2,W).\] 
		To see that the resulting morphism lands in $\Gr_{\Gr(4,V_5)}(k+2,\cW)$ is to ask that the composition 
		\[\cO_{Y}(1) \xrightarrow{\textbf{s}'} V_5^\vee \to V_5 \otimes \bigwedge^2 V_5^\vee \to V_5 \otimes \cU_{k+2}^\vee\]
		vanishes, where the second map is coevaluation and the third is restriction. Working locally on $U$, we may pick an open cover $\bigcup_i U_i$ which will trivialize $\cO_Y(1)$ and $\coker(\textbf{s}')$ so that $V_5^\vee$ admits a basis of sections $x_1, \dots, x_4, x_5$ over $U_i$ where $x_1, \dots, x_4$ are a basis for $\coker(\textbf{s}')|_{U_i}$ and $x_1$ is a basis for $\cO_Y(1)$. Writing $x_i \wedge x_j$ as $x_{ij}$, the map $V_5 \to \Sym^2 \cU_{k+2}^\vee$ is then given by the restriction of the Pl\"ucker quadrics 
		\begin{equation} \label{pluckerrestrictions}
			\begin{gathered}
				x_{23}x_{45} - x_{24}x_{35} + x_{25}x_{34} \\
				x_{13}x_{45} - x_{14}x_{35} + x_{15}x_{34} \\
				x_{12}x_{45} - x_{14}x_{25} + x_{15}x_{24} \\
				x_{12}x_{35} - x_{13}x_{25} + x_{15}x_{23} \\
				x_{12}x_{34} - x_{13}x_{24} + x_{14}x_{23}
			\end{gathered}
		\end{equation}
		to $\cU_{k+2}$, while the composition $\cO_Y(1) \to V_5 \otimes \cU_{k+2}^\vee$ gives equations 
		\begin{equation} \label{otherrestrictions}
			x_{15}, x_{25}, x_{35}, x_{45}
		\end{equation}
		on $\cU_{k+2}$. But it is a straightforward check that if the last equation of \eqref{pluckerrestrictions} does not vanish but the first four equations do, the equations \eqref{otherrestrictions} must also vanish. In fact, this is classical, and amounts to the assertion that for any $U_4 \subset V_5$, we have an identity of schemes 
		\[\bigcap_{v \in U_4} Q_u = \Gr(2,V_5) \cup \bbP(\bigwedge^2 U_4)\]
		where $Q_v$ is the Pl\"ucker quadric associated to $v \in V_5$. \par 
		The check that the compositions $U \to U' \to U$ and $U' \to U \to U'$ yield the identity is similar to the case of the intersection of two quadrics, so we omit it. \par 
		We are left only to check that the codimension of $Z'$ is equal to rank $\Sym^2 \cU_{k+2}^\vee$ when $k = 0$ or $1$. However, the arguments of Lemma \ref{quinticprojbundlelemma} show that $\dim Z' = \dim F_{k+1}^\sigma(X) + 1 - k$, so it suffices to check that this quantity agrees with
		\begin{equation}
				\dim Y' - \on{rank} \Sym^2 \cU_{k+2}^\vee = (n-k-2)(k+2) + 4 - {k+3 \choose 2}.
		\end{equation}
	 	When $k = 0$ this can be rewritten as $2n - 3$, and when $k = 1$ it becomes $3n - 11$, which one can check by hand using Proposition \ref{fanodims} to agree with $\dim F_{k+1}^\sigma(X) + 1 - k$. 
	\end{proof}
	\begin{proof}[Proof of Theorem \ref{geomflipquintic}]
		Combine Lemma \ref{flippinglemma} with Lemmas \ref{smoothnesslemmaquintics} and \ref{sectionlemmaquintics}. 
	\end{proof}
	\section{An application to semiorthogonal decompositions for orthogonal Grassmannian fibrations} \label{application}
	As a key example, let us specialize to the case where $k = 0$ and $X \subset \bbP^{n+2}$ is a complete intersection of two quadrics where $n$ is odd. In this case, our semiorthogonal decomposition from Corollary \ref{sodcorollary} reduces to the statement 
	\begin{equation}\label{quadfibsod}
		\Db(X^{[2]}) = \langle\Db(\OGr(2,\fQ)), \Db(F_1(X))\rangle.
	\end{equation}
	In fact $\Db(X)$ admits a semiorthogonal decomposition which can be written as 
	\begin{equation}
		\Db(X) = \langle \Db(C), \cO, \dots, \cO(n-2) \rangle 
	\end{equation}
	where $C$ is the hyperelliptic curve branched over the singular quadrics in the pencil spanned by the two quadrics \cite[Theorem 2.9]{bondalorlov}. On the other hand, applying \cite[Theorem 3.1]{krug} to the double cover $\Bl_\Delta(X \times X)$ and using the blowup decomposition, we have a decomposition
	\begin{equation}
		\Db(X^{[2]}) = \langle \Db_{\bbZ/2}(X \times X), \underbrace{\Db(X),\dots,\Db(X)}_{n-2\text{ times}}, \rangle 
	\end{equation}
	and using \cite[Theorem 1.1]{koseki} we find 
	\begin{equation}
		\begin{aligned}
			\Db_{\bbZ/2}(X \times X) &\simeq \Sym^2 \Db(X) \\
			&= \langle \Sym^2 \Db(C), \underbrace{\Db(C), \dots, \Db(C)}_{n-1 \text{ times}}, \underbrace{\Db(\bbC), \dots, \Db(\bbC)}_{{n-1 \choose 2}\text{ times}}, \underbrace{\Sym^2 \Db(\bbC), \dots, \Sym^2 \Db(\bbC)}_{n-1\text{ times}}\rangle \\
			&= \langle \Db(\Sym^2 C), \underbrace{\Db(C), \dots, \Db(C)}_{n\text{ times}}, \underbrace{\Db(\bbC),\dots,\Db(\bbC)}_{{n-1 \choose 2} + 2(n-1)\text{ times}}\rangle
		\end{aligned}
	\end{equation}
	where we have applied the equivalence for a root stack 
	\[\Sym^2 \Db(C) \simeq \Db_{\bZ/2}(C \times C) = \langle\Db(\Sym^2 C),\Db(C)\rangle\]
	and 
	\[\Sym^2 \Db(\bbC) = \Db_{\bZ/2}(\bbC) = \langle \Db(\bbC), \Db(\bbC)\rangle.\]
	Combining the statements above, we find 
	\begin{equation}
		\Db(X^{[2]}) = \langle \Db(\Sym^2 C), \underbrace{\Db(C),\dots,\Db(C)}_{2n-2\text{ times}}, \underbrace{\Db(\bbC),\dots,\Db(\bbC)}_{{n-1\choose 2} + 2(n-1) + (n-1)(n-2)\text{ times}}\rangle.
	\end{equation}
	On the other hand, specializing the main conjecture of \cite{fanoschemequaddecomp} to the case of $\Db(F_1(X))$ yields the following conjecture. 
	\begin{conjecture}
		\label{fanoconjecture}
		Let $X = Q_1 \cap Q_2 \subset \bbP^{n+2}$ be a smooth complete intersection of two quadrics in an odd-dimensional projective space. Then there is a semiorthogonal decomposition 
		\begin{equation}
			\label{fanosod}
			\Db(F_1(X)) = \langle\Db(\Sym^2 C), \underbrace{\Db(C),\dots,\Db(C)}_{n - 3 \text{ times}}, \underbrace{\Db(\bbC),\dots,\Db(\bbC)}_{{n-4\choose 2} + 2(n-4) \text{ times}}\rangle. 
		\end{equation}
	\end{conjecture}
	This motivates a natural conjecture for a semiorthogonal decomposition of $\Db(\OGr(2,\fQ))$ given by ``subtraction'' of the semiorthogonal decomposition \eqref{fanosod} from \eqref{quadfibsod}. 
	\begin{conjecture}
		\label{pencilconjecture}
		Let $\langle Q_1, Q_2 \rangle$ be a pencil of quadrics in $\bbP^{n+2}$ whose base locus is smooth in an odd-dimensional projective space. Then there is a semiorthogonal decomposition 
		\begin{equation}
			\Db(\OGr(2,\fQ)) = \langle \underbrace{\Db(C), \dots, \Db(C)}_{n+1\text{ times}}, \underbrace{\Db(\bbC), \dots, \Db(\bbC)}_{(n-1)(n+1)\text{ times}}\rangle.
		\end{equation}
	\end{conjecture}
	Formally, one can interpret the relation between the two Conjectures \ref{fanoconjecture} and \ref{pencilconjecture} as giving equivalent identities in the Grothendieck ring of categories \cite{grothring}. In fact, this conjecture should be interpreted as a relativization of the full exceptional collections for Grassmannians of isotropic lines in \cite{kuznetsoveven}, and admits a natural generalization to any family of quadrics, though we impose a mild condition on the rank of the quadrics to ensure the fibration of orthogonal Grassmannians is flat. 
	\begin{conjecture}
		\label{cliffordconjecture}
		Let $p : \fQ \to S$ be any quadric fibration of relative dimension $n + 1$ over a smooth base such that the fibers $\fQ_s$ are quadrics of rank at least 4, where $n$ is odd. Then there exists a semiorthogonal decomposition 
		\begin{equation}
			\Db(\OGr(2,\fQ)) = \langle \underbrace{\Db(S,\cl_0), \dots, \Db(S,\cl_0)}_{n+1\text{ times}}, \underbrace{\Db(S), \dots, \Db(S)}_{\frac{(n-1)(n+1)}{2}\text{ times}}\rangle,
		\end{equation}
		where $\Db(S,\cl_0)$ is the derived category of the sheaf of even Clifford algebras over the base $S$ associated to the quadric fibration $\fQ\to S$. 
	\end{conjecture}
	\begin{remark}
		\begin{itemize}
			\item When $S = \bbP^1$ and $\fQ \to S$ is a nontrivial pencil of quadrics with smooth base locus, this reduces to Conjecture \ref{pencilconjecture} by the central reduction in \cite{quadfibs}. We use also that $\Db(\bbP^1)$ has two exceptional objects. 
			\item When $S$ is a point and $\fQ\to S$ is a single smooth quadric in an odd-dimensional projective space, this reduces to a full exceptional collection as in \cite{kuznetsoveven}, as in this case $\Db(S,\cl_0) \simeq \Db(\bbC \sqcup \bbC)$. 
		\end{itemize}
	\end{remark}
	We will give partial progress on this conjecture in forthcoming work. \par 
	Note that while a similar calculation can be done for quadrics in odd-dimensional projective spaces, the analogous conjecture is less clear. 
	\section{A counterexample for low degree del Pezzo varieties} \label{counterexample}
	Even in the cases where the del Pezzo variety has degree $d = 1$ or $2$, it is still possible to make sense of the Fano scheme of lines for $d \leq 2$ by defining it as the Hilbert scheme of smooth rational curves with the appropriate Hilbert polynomial. It is reasonable to ask whether some version of Theorems \ref{delpezzoflip} or \ref{geomflipthm} may hold in this situation. However, one should observe that the argument used in Theorem \ref{delpezzoflip} relies on the fact that the natural polarization for the del Pezzo variety $X$ is very ample; in particular, the natural map $\bbP(\Sym^2 \cU_2^\vee) \to X^{[2]}$ fails to be a closed immersion unless $d \geq 3$. Nonetheless, one might hope that the decomposition on the level of derived categories holds. \par 
	In the case for $d = 2$ and dimension 3, the del Pezzo $X$ is the quartic double solid, and when the branching quartic contains no lines, the variety $F_1(X)$ is a surface whose cohomological invariants have been studied in detail \cite{welters}. In this case, it is possible to show that Theorem \ref{delpezzoflip} cannot possibly hold. 
	\begin{proposition}
		Suppose that $X$ is a general del Pezzo variety of dimension $3$ and degree $2$, so that $X$ is a quartic double solid branched along a quartic K3 containing no lines. Then $\Db(F_1(X))$ is not a semiorthogonal component of $\Db(X^{[2]})$.
	\end{proposition}
	\begin{proof}
		We obstruct the existence of this semiorthogonal decomposition by examining additive invariants of the categories $\Db(F_1(X))$ and $\Db(X^{[2]})$. In particular, recall that the Hochschild homology is an additive invariant under semiorthogonal decomposition, so the existence of an embedding $\Db(F_1(X)) \hookrightarrow \Db(X^{[2]})$ would imply an inequality $\dim HH_0(F_1(X)) \leq \dim HH_0(X^{[2]})$. On the other hand, \cite[Cohomological study]{welters} calculates the cohomological invariants \[h^0(F_1(X), \cO) = 1 \text{ and } h^1(F_1(X), \Omega_{F_1(X)}) = 220,\]  while Serre duality gives $h^2(F_1(X), \omega_{F_1(X)}) = 1$. However, the Hochschild--Kostant--Rosenberg theorem implies that $HH_0$ agrees with the sum of the Hodge numbers of type $(p,p)$, so in particular 
		\[\dim HH_0(F_1(X)) = h^0(\cO_{F_1(X)}) + h^1(\Omega_{F_1(X)}) + h^2(\omega_{F_1(X)})= 222.\]  \par 
		To calculate $HH_0(X^{[2]})$, by \cite[Section 0]{welters} we know that the bottom half of the Hodge diamond of $X$ is given by 
		\begin{gather*}
			0 \quad 10 \quad 10 \quad 0 \\
			0 \quad 1 \quad 0 \\
			0 \quad 0 \\
			1
		\end{gather*}
		Since $X^{[2]}$ is the quotient of the blowup $\Bl_{\Delta} (X \times X)$ along the diagonal by the natural involution, by the blowup formula for cohomology, we see that that there is a natural identification of cohomology with complex coefficients given by 
		\begin{equation}
			\begin{aligned}
				H^*(X^{[2]}) &= \left(H^*(X) \otimes H^*(X) \oplus H^*(X)E \oplus H^*(X)E^2\right)_{\bbZ/2} \\
				&= (H^*(X) \otimes H^*(X))_{\bbZ/2} \oplus H^*(X)E \oplus H^*(X)E^2
			\end{aligned}
		\end{equation}
		where $E$ is the class of the exceptional divisor for $\Bl_\Delta X \times X$. As this naturally respects Hodge decompositions, we observe that there are isomorphisms
		\begin{equation}
			\begin{aligned}
				H^{0,0}(X^{[2]}) &= \Sym^2 H^{0,0}(X) \\
				H^{1,1}(X^{[2]}) &= H^{0,0}(X) \otimes H^{1,1}(X) \oplus H^{0,0}(X)\\
				H^{2,2}(X^{[2]}) &= H^{0,0}(X) \otimes H^{2,2}(X) \oplus \Sym^2 H^{1,1}(X) \oplus H^{1,1}(X) \oplus H^{0,0}(X) \\
				H^{3,3}(X^{[2]}) &= H^{0,0}(X) \otimes H^{3,3}(X) \oplus H^{1,1} \otimes H^{2,2} \oplus H^{2,1} \otimes H^{1,2} \oplus H^{2,2}(X) \oplus H^{1,1}(X).
			\end{aligned}
		\end{equation}
		We conclude that the center vertical strip of the Hodge diamond of $X^{[2]}$ is given by 
		\[1\quad 2 \quad 4 \quad 104 \quad 4 \quad 2 \quad 1,\]
		and after applying Hochschild--Kostant--Rosenberg once more find
		\[\dim HH_0(X^{[2]}) = 118.\]
		As 
		\[\dim HH_0(F_1(X)) > \dim HH_0(X^{[2]}),\]
		we conclude that there can be no semiorthogonal decomposition of $X^{[2]}$ including $F_1(X)$ as a component. 
	\end{proof}
	This is the lowest dimension where such an obstruction exists for degree 2 del Pezzo varieties, since a del Pezzo surface $S$ of degree 2 has 56 lines and so $\Db(F_1(S))$ will have 56 exceptional objects, while $\Db(S^{[2]}) \simeq \Sym^2 \Db(S)$ will have $10\cdot 9/2 + 2\cdot 10 = 65$ exceptional objects, coming from symmetrizing the length 10 exceptional collection of $\Db(S)$. \par 
	An interesting line of investigation would be to check whether there is some connection on the level of derived categories between the geometry of $F_1(X)$ and $X^{[2]}$ up to a natural birational modification or \'etale cover. 
	\printbibliography
\end{document}